\documentclass[a4paper,12pt]{article}
\usepackage{hyperref}
\usepackage{pifont}
\usepackage{amsmath,amsthm,amssymb,amsfonts}
\usepackage{longtable}
\usepackage{multirow}
\usepackage{leftidx}
\usepackage[noend]{algpseudocode}
\usepackage{algorithmicx,algorithm}
\usepackage{bbm,bm}
\usepackage{latexsym}
\usepackage{mathrsfs}
\usepackage{threeparttable}
\usepackage{tabularx}
\usepackage{booktabs}
\usepackage{multicol}
\usepackage{subfigure}
\usepackage{graphicx}
\usepackage{rotating}
\usepackage{epstopdf}
\usepackage{color}
\usepackage{epsfig}
\usepackage{indentfirst}
\usepackage{geometry}
\usepackage{caption}
\usepackage{lineno}
\usepackage{enumerate,mdwlist}
\usepackage[numbers,sort&compress]{natbib}  
\usepackage[utf8]{inputenc}
\usepackage{float}
\theoremstyle{definition}
\newtheorem{The}{Theorem}[section]    
\newtheorem{Lem}[The]{Lemma}
\newtheorem{Cor}[The]{Corollary}
\newtheorem{Pro}[The]{Proposition}

\newtheorem{Cla}{Claim}

\newtheorem{Cas}{Case}

\newtheorem{Rem}[The]{Remark}

\theoremstyle{definition}

\numberwithin{equation}{section}

\begin{document}
\title{Excluded conformal minors of Birkhoff-von Neumann graphs with equal global forcing number and
maximum anti-forcing number
\footnote{This work is supported by NSFC (Grant No. 12271229) and the Scientific Research Startup Fund of Sichuan Normal University (Grant No. kyqd20260308)}
}
\author{Yaxian Zhang$^{1,2}$, Yan Wu$^1$ and Heping Zhang$^{1}$\thanks{Corresponding author.}}
\date{
   {\small $^1$School of Mathematics and Statistics, Lanzhou University, Lanzhou, Gansu 730000, P.R. China\\
           $^{2}$School of Mathematical Sciences, Sichuan Normal University, Chengdu, Sichuan 610068, P.R. China
   }\\
   {\small E-mails:\ yaxianzhang@sicnu.edu.cn, wuyan@emails.bjut.edu.cn, zhanghp@lzu.edu.cn}
}

\maketitle

\begin{abstract}
  Global forcing number and maximum anti-forcing number of matchable graphs (graphs with a perfect matching) were proposed in completely different situations with applications in theoretical chemistry.  Surprisingly for bipartite graphs and  some nonbipartite graphs as solid bricks (or  Birkhoff-von Neumann graphs) $G$,  the global forcing number $gf(G)$  is at least  the maximum anti-forcing number $Af(G)$. It is natural to consider when $gf(G)=Af(G)$ holds. For convenience, we call a matchable graph $G$ {\em strongly uniform} if each conformal matchable subgraph $G'$ always satisfies $gf(G')=Af(G')$.
  In this article, by applying the ear decomposition theorem and discussing the existence of  a Hamilton cycle with positions of chords,
  we give  ``excluded conformal minors'' and  ``structural'' characterizations of matchable bipartite graphs and
  Birkhoff-von Neumann graphs that are strongly uniform respectively.


    \linebreak Birkhoff-von Neumann graph; Conformal minor; Hamilton cycle
    \medskip
\end{abstract}
\section{Introduction}
A perfect matching (or 1-factor) of a graph is a fundamental notion in graph theory, which coincides with a  \emph{Kekul\'e structure}
in organic chemistry \cite{CG} and a \emph{dimer configuration} in statistical physics \cite{Yan2008,Fisher1961,Kasteleyn1961}.
Many graphs have often  an enormous number of perfect matchings. For example, Esperet et al. \cite{EK} showed that  a cubic bridgeless graph has an  exponential many perfect matchings  with its order, confirming an old conjecture of Lov\'asz  and Plummer in 1970's.
Vuki\v cevi\'c and  Sedlar \cite{vukivcevic2004total} introduced the \emph{global} (or \emph{total}) \emph{forcing number} of a  graph $G$ as a graph parameter denoted by $gf(G)$, i.e.,  the smallest cardinality of an edge subset distinguishing all perfect matchings of $G$. Dosli\'c \cite{D.T2007Global} showed that this number cannot exceed the cyclomatic number (see Theorem \ref{gf(G)<c(G)}).
Forcing and anti-forcing ideas for a perfect matching of a graph originate from resonant structures in theoretical chemistry (cf.  Klein and Randi\'c \cite{klein1987innate}, and Li \cite{Li1997Hexagonal} and Vuki\v{c}evi\'c and Trinajsti\'c \cite{vukiveevic2007anti}). For a recent survey on this direction, see \cite{ZHLZ2025}.

An edge $e$ of $G$ is a \emph{forcing single edge} (or an \emph{anti-forcing edge}) if $G$ has a unique perfect matching avoiding $e$.
Generally, Lei et al. \cite{Lei2016Anti} and Klein and Rosenfeld \cite{Klein2014} independently defined the
\emph{anti-forcing number} of a single perfect matching $M$ of a graph $G$ as the smallest cardinality of an
edge subset whose removal results in a graph with  a unique  perfect matching $M$.
 The largest value over anti-forcing numbers of all perfect matchings of a graph $G$, called the {\sl maximum anti-forcing number} of $G$ and denoted by $Af(G)$, is significant in theoretical chemistry. For example, the maximum anti-forcing numbers of a hexagonal system and (4,6)-fullerene are equal to their Fries numbers, respectively (cf. \cite{Lei2016Anti,shi2017maximum}), which can measure the stability of the corresponding molecules \cite{Fries1927}.

 Surprisingly the first and third authors  \cite{Zhang2022} discovered a close connection between the global forcing number and the maximum anti-forcing number of a matchable bipartite graph.
 A graph $G$ is called {\sl matchable} if it has a  perfect matching. 

 \begin{The}\label{bipartite}
   {\rm \cite{Zhang2022}}
   For a matchable bipartite graph $G$,  $gf(G)\geq Af(G)$.
\end{The}

Such a connection also exists in many nonbipartite graphs.  The {\sl perfect matching polytope} $PM(G)$ of a  matchable graph $G$ is  the set of all convex linear combinations of incidence vectors of perfect matchings of $G$ in $\mathbb{R}^{E(G)}$. Edmonds \cite{edmonds1965maximum} showed that $PM(G)$ can be described by the set of  vectors $x\in \mathbb{R}^{E(G)}$ satisfying following linear inequalities:
\begin{center}\begin{description}
   \item[(1)] $x\geq 0$,
   \item[(2)] $x(\partial(v))=1, \mbox{for each vertex } v\in V(G)$,
   \item[(3)] $x(\partial(S))\geq 1, \mbox{for each odd subset $S$ of } V(G).$
 \end{description}
\end{center}
Then $G$   is called \emph{Birkhoff-von Neumann graph} (or BN-graph, simply)  if $PM(G)$ is characterized only by nonnegativity (1) and degree constraints (2) in the above inequalities. Due to the classical results of Birkhoff \cite{Birkhoff1976} and von Neumann \cite{von1953},
BN-graphs include all bipartite graphs. Balas \cite{Balas1981} gave a combinatorial characterization  for  BN-graphs.
Reed, Wakabayashi, and Carvalho, Lucchesi, Murty \cite{Car} respectively showed that a brick is BN if and only if it is solid in different methods.
%
A pair of disjoint odd cycles $(C_1,C_2)$ in $G$ is an \emph{odd conformal bicycle} if $G-V(C_1)-V(C_2)$ is matchable.

\begin{The}
   {\rm\cite{Balas1981}}
   \label{EBV}
   A graph $G$ is a BN-graph if and only if
   $G$ does not contain odd conformal bicycles.
 \end{The}

 Theorem \ref{bipartite} also holds for  BN-graphs.

\begin{The}
   {\rm \cite{Zhang2022}}
   \label{Gf>Af}
   Let $G$ be a matchable BN-graph.
   Then $gf(G)\geq Af(G)$.
\end{The}

It is natural to wonder when the equality holds. However it is quite difficult to determine all BN-graphs with the equality.
Similar to the definitions of perfect graphs \cite{CCLSV2005} and  digraph packing \cite{GT2011}, we give  further requirements on all matchable conformal subgraphs. A subgraph $H$ of a matchable graph $G$ is {\em conformal} if $G-V(H)$ is matchable \cite{CLM2005}. Conformal subgraphs are also called {\em nice} or {\em central} subgraphs \cite{RST1999}. We define a matchable graph $G$ to be \emph{strongly uniform graph} if for each conformal matchable subgraph $H$ of $G$,  $gf(H)=Af(H)$ always holds. A graph $J$ is called a {\em conformal minor} of a matchable graph $G$ if some bisubdivision $H$ of $J$ (defined in Section \ref{DeBi}) is a conformal subgraph of $G$. In this paper we give  ``excluded conformal minors'' and  ``structural'' characterizations of matchable bipartite graphs and
  BN-graphs that are strongly uniform respectively.

Inspired by the idea of excluding minors in \cite{NT2007,Little1974,FL2001,Chen20201}, we want to find some key conformal minors not meeting the equality.
We first find 4 non-strongly uniform bipartite graphs by bipartite ear decomposition.
Excluding these graphs as conformal minors,
a matchable bipartite graph must contain a \emph{Hamilton cycle} (a cycle covering all vertices).
Based on the Hamilton cycle, we completely characterize strongly uniform matching covered
bipartite graphs (see Theorem \ref{BMT} in Section \ref{Bip-C}).
Except for the 4 aforementioned bipartite graphs,
we find another 25 non-strongly uniform BN-graphs by graded ear decomposition.
By excluding the above 29 graphs as conformal minors, we obtain our main theorem of this article: a BN-graph is strongly uniform if and only if it contains no
conformal minor isomorphic to the above 29 graphs (see Theorem \ref{odd-pre} in Section 4). To get the main theorem, we determine 4 families of  BN-graphs by excluding the 29 conformal minors and prove that each graph $G$ in the 4 families of BN-graphs satisfies that $gf(G)=Af(G)$, which are described in Lemmas \ref{BNFG} and \ref{G1234E} respectively. We will give their proofs in Sections 5 and 6. 

\section{Preliminaries}

We restrict our consideration to simple, undirected and finite graphs.
For the terminology on matchable graphs, we follow Lov\' asz and Plummer \cite{Lovasz1986}.
 Let $G$ be a graph with vertex set $V(G)$ and edge set $E(G)$.
 An edge subset $M$ of $G$ is a \emph{matching}
 if any pair of edges in $M$ has no common end-vertices.
A \emph{perfect matching} $M$ of $G$ is a matching covering all vertices of $G$, i.e., $V(M)=V(G)$.
A cycle $C$ of $G$ is $M$-alternating if $M$ contains a perfect matching of $C$. So a cycle of $G$ is conformal if and only if $G$ has a perfect matching $M$ such that $C$ is $M$-alternating.
In the following, we use $\mathcal{M}(G)$ to denote the set of all perfect matchings of $G$.

A connected  graph $G$ is
\emph{matching covered} if  each edge  of $G$ belongs to a perfect matching. An edge $e$ of a matchable graph $G$ is called \emph{allowed}, if it belongs to at least one
perfect matching of $G$, and \emph{forbidden}  otherwise. Each component of the subgraph of a matchable graph $G$ formed by all allowed edges is matching covered and  called an {\sl elementary component} of $G$.

\subsection{Anti-forcing and global forcing}

For a perfect matching $M$ of a graph $G$,
an \emph{anti-forcing set} of $M$ means an edge subset $S\subseteq E(G)\setminus M$ such that
$G-S$ has only one perfect matching $M$, and the \emph{anti-forcing number} of $M$ is the
smallest size of an anti-forcing set of $M$, denoted by $af(G,M)$.

\begin{Lem}
   {\rm\cite{Kai2017Anti}}
   \label{anti-forcing}
    Let $M$ be a perfect matching of a graph $G$.
    Then an edge subset $S\subseteq E(G)\setminus M$ is an anti-forcing set of $M$ if and only if
    $S$ contains at least one edge from each $M$-alternating cycle of $G$.
\end{Lem}

A \emph{compatible $M$-alternating set} of a graph $G$
is a set of $M$-alternating cycles $\mathcal C$
satisfying that each pair of cycles in $\mathcal C$ is either disjoint or shares only edges in $M$.
The maximum cardinality of a compatible $M$-alternating set of $G$ is denoted by $c'(G,M)$.

\begin{The}
   \rm{\cite{Lei2016Anti}}
   \label{com-alt}
   For any perfect matching $M$ of a plane bipartite graph $G$, we have $af(G,M)=c'(G,M)$.
\end{The}

A \emph{global forcing set} of a matchable graph $G$ is an edge subset
$S$ such that for each pair of distinct perfect matchings $M_1$ and $M_2$ of $G$, $S\cap M_1\neq S\cap M_2$. The  \emph{global forcing number} of $G$ is the smallest cardinality
of a global forcing set of $G$, denoted by $gf(G)$.

\begin{Lem}
   {\rm\cite{cai2012global}}
   \label{global-forcing}
   For an edge subset $S$ of a matchable graph $G$, $S$ is a global forcing set of $G$ if and only if $S$ intersects each conformal cycle of $G$.
\end{Lem}

We use $c(G)$ to denote the \emph{cyclomatic number} of a connected graph $G$, i.e., $c(G)=|E(G)|-|V(G)|+1$.
Do\v sli\'c \cite{D.T2007Global} gave a trivial upper bound of the global forcing number.

\begin{The}
   {\rm\cite{D.T2007Global}}
   \label{gf(G)<c(G)}
   Let $G$ be a connected matchable graph. Then $gf(G)\leq c(G)$ and the equality holds
   if and only if all cycles of $G$ are conformal.
\end{The}

Recall that a matchable graph $G$ is \emph{strongly uniform} if for each conformal matchable subgraph $H$ of $G$,  $gf(H)=Af(H)$ always holds.

\subsection{Bisubdivision\label{DeBi}}
Let $S$ be an edge subset of a graph $G$.
If $G'$ is obtained from $G$ by replacing each edge $e$ in $S$ with a path $P_e$ of odd length
 and these paths are internally disjoint,
then we call $G'$ a \emph{bisubdivision} of $G$ on $S$, and $P_e$ a \emph{replacement-path} of $e$.

Then we can establish a bijection $g$ between $\mathcal{M}(G)$ and $\mathcal{M}(G')$:
For each $M\in \mathcal{M}(G)$, $g(M)=(M\setminus S)\cup M_S$,
where $M_S$ consists of the $i$-th edges of a path $P_e$ with odd $i$ for all $e\in M\cap S$ and
the $i$-th edges of a path $P_e$ with even $i$ for all $e\in S\setminus M$.
We consider possible changes of $gf(G)$ and $Af(G)$ after some bisubdivisions.
Zhao and Zhang \cite{ZS2019} have already got the following.

\begin{Lem}
   {\rm \cite{ZS2019}}
   \label{ES-A}
   Let $G'$ be a bisubdivision of a matchable graph $G$.
   Then $af(G',g(M))\leq af(G,M)$ for a perfect matching $M$ of $G$.
   Moreover, $Af(G')\leq Af(G)$.
\end{Lem}


\begin{Lem}
   \label{Ev-NM}
   Let $M$ be a perfect matching of a graph $G$ with $af(G,M)=Af(G)$ and $e\in E(G)\setminus M$.
   If $G'$ is a bisubdivision of $G$ on $e$, then $Af(G')=Af(G)$.
\end{Lem}

\begin{proof}
   By Lemma \ref{ES-A}, $Af(G')\leq Af(G)$.
   For the graph $G'$ and its perfect matching $g(M)$,
   let $S'$ be a minimum anti-forcing set of $g(M)$.
   For the replacement-path $P_e$ of $e$, $S'\cap E(G)=S'\setminus E(P_e)$.
   Let
   $S=(S'\setminus E(P_e))\cup \{e\}$, if the $i$-th edges of $P_e$, for some odd $i$, belongs to $S'$;
   $S=S'\setminus E(P_e)$, otherwise.
   By Lemma \ref{anti-forcing} $S$ is an anti-forcing set of $M$. So $Af(G')\geq af(G',g(M))=|S'|\geq |S|\geq af(G,M)=Af(G)$ and thus $Af(G')=Af(G)$.
\end{proof}
The above discussions imply that a bisubdivision of a graph does not increase the maximum anti-forcing number.
However we will get that bisubdivisions of a graph remain unchanged  on the global forcing number.

\begin{Lem}
   \label{ES}
   For a bisubdivision $G'$ of a matchable graph $G$,  $gf(G')=gf(G)$.
\end{Lem}
\begin{proof}
   Without loss of generality, assume $G'$ is obtained from $G$ by replacing an edge $e$ with a path of odd length $P_e$.
   Let $S$ be a global forcing set of $G$. Let
   $S':=(S\setminus \{e\})\cup \{e'\}$ for the first edge $e'$ of $P_e$, if $e\in S$, and $S':=S$, otherwise.
   For each conformal cycle $C'$ of $G'$,
   we can obtain a conformal cycle $C$ of $G$ by replacing $P_e$ with $e$ (if $P_e\subseteq C'$).
   By Lemma \ref{global-forcing}, $S\cap E(C)\neq \emptyset$, which is equivalent to that $S'\cap E(C')\neq \emptyset$.
   Once by Lemma \ref{global-forcing}, $S'$ is a global forcing set of $G'$. The above discussion is reversible. So $gf(G')=gf(G)$.
\end{proof}
\subsection{\label{QS}Quadrilateral subdivision}
For a graph $G$ and an edge subset $S\subseteq E(G)$,
let $G'$ be the graph obtained from $G$ by replacing each edge $e_i=u_iw_i$ in $S$ with a
$(4k_i+1)$-length path $P_{e_i}'=u_iv_1^iv_2^iv_3^iv_4^i\ldots v_{4k_i}^iw_i$ and
adding $k_i$ new edges $v_1^iv_4^i,v_5^iv_8^i,\ldots v_{4k_i-3}^iv_{4k_i}^i$, where $1\leq i\leq |S|$,
$k_i\geq 0$ and these paths are internally disjoint.
If $G''$ is a bisubdivision of $G'$ on $E(P_{e_i}')$,
then $G''$ is called a \emph{quadrilateral subdivision} of $G$ on $S$ and each $P_{e_i}'$ is called also a {\it replacement-path} of $e_i$.
For example,   quadrilateral subdivisions of $H_{1,4}$ (see Fig. \ref{F10})
on edges $\{v_1v_2,v_3v_4\}$ give a graph $H_{1,4}'$, see Fig. \ref{F12-2}.

\begin{figure}[H]
   \vspace{3pt}
   \centerline{\includegraphics[width=0.17\textwidth]{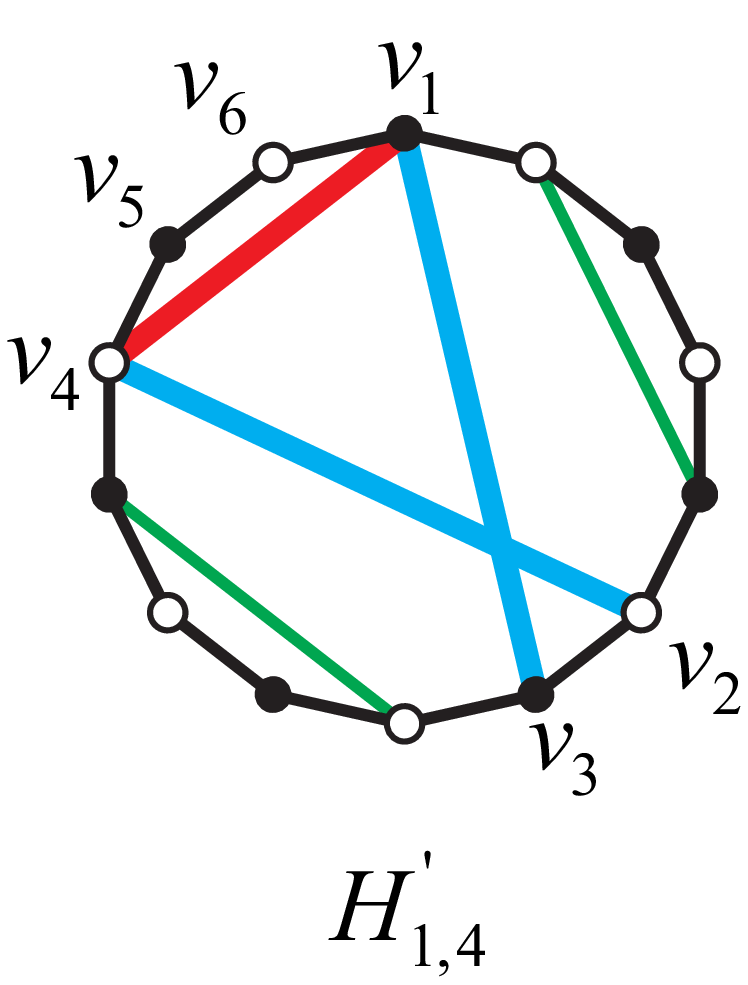}}
\caption{\label{F12-2} Quadrilateral subdivisions of $H_{1,4}$  on edges $\{v_1v_2,v_3v_4\}$ to get graph $H_{1,4}'$.}
\end{figure}
\begin{Lem}
      \label{Ev-Ad} Let $M$ be a perfect matching of a connected graph $G$ with $af(G,M)=Af(G)$. For an edge $e\in E(G)\setminus M$, let $G_e$ be a graph obtained from  $G$ by replacing $e=uv$ with a path
      $P_e=uw_1w_2w_3w_4v$ and adding a new edge $w_1w_4$.
      Then $gf(G_e)=gf(G)+1$ and $Af(G_e)=Af(G)+1$.
   \end{Lem}

   \begin{proof}
      Since $e\notin M$, $C=w_1w_2w_3w_4w_1$ is a conformal square of $G_e$.
      For each minimum global forcing set $S_e$ of $G_e$,
      $|S_e\cap E(C)|\geq 1$ by Lemma \ref{global-forcing}.
      Note that $w_2$ and $w_3$ have degree two in $G_e$.
      Suppose that $G_e-S_e$ is not connected. Since $G$ is connected, $G_e$ is also connected. Then there is at least one edge $e'\in S_e$ connecting two components of $G_e-S_e$.
      By Lemma \ref{global-forcing}, $G_e-S_e$ contains no conformal cycles of $G_e$. It follows that $G_e-(S_e\setminus \{e'\})$ contains no conformal cycles of $G_e$. Once by Lemma \ref{global-forcing}, $S_e\setminus \{e'\}$ is a global forcing set of $G_e$ with less size than $S_e$, a contradiction.
      Thus, $G_e-S_e$ is connected.
      If $|S_e\cap E(C)|\geq 2$, then
      $S_e$ contains the edge $w_1w_4$ and exactly one edge
      from $P_e$,
      i.e., $|S_e\cap E(C)|=2$ and $S_e\cap \{uw_1,w_4v\}=\emptyset$.
      By Lemma \ref{global-forcing},
      $(S_e\setminus E(C))\cup \{uw_1,w_1w_2\}$ is also a minimum global forcing set of $G_e$.
      Thus, we can assume that $|S_e\cap E(C)|=1$.

      Note that each conformal cycle of $G$ corresponds to a conformal cycle of $G_e$ intersecting $S_e$.
      Let $S=S_e\setminus E(C)$, if $S_e\cap \{uw_1,w_4v\}=\emptyset$;
      $S=((S_e\setminus E(C))\setminus \{uw_1,w_4v\})\cup \{e\}$, otherwise.
      Then we can check that $S$ is a global forcing set of $G$, which implies that
      $gf(G_e)=|S_e|\geq 1+|S|\geq 1+gf(G)$.
      Since each conformal cycle of $G_e$, other than $C$, corresponds to
      a conformal cycle of $G$ and the above discussion is reversible,
      $gf(G_e)\leq gf(G)+1$. So $gf(G_e)=gf(G)+1$.

      Let $M_e=M\cup \{w_1w_4,w_2w_3\}$.
      Then $M_e$ is a perfect matching of $G_e$ and $C$ is an $M_e$-alternating square.
      Let $S'_e$ be a minimum anti-forcing set of $M_e$ in $G_e$.
      Then $S'_e$ contains exactly one edge of $E(C)$.
      Let $S'=S'_e\setminus E(C)$, if $S'_e\cap \{uw_1,w_4v\}=\emptyset$;
      $S'=((S'_e\setminus E(C))\setminus \{uw_1,w_4v\})\cup \{e\}$, otherwise.
      Then $S'$ is an anti-forcing set of $M$ in $G$, which implies that
      $Af(G_e)\geq af(G_e,M_e)=|S'_e|\geq 1+|S'|\geq 1+af(G,M)=1+Af(G)$.

      Let $M'$ be a perfect matching of $G_e$.
      If $w_1w_4\notin M'$, then
      $af(G_e,M')\leq af(G_e-w_1w_4,M')+1\leq Af(G_e-w_1w_4)+1$. By Lemma \ref{ES-A}, $Af(G_e-w_1w_4)\leq Af(G)$ and thus $af(G_e,M')\leq Af(G)+1$.
      If $w_1w_4\in M'$, then $w_2w_3\in M'$ and $C$ is an $M'$-alternating square.
      Like the above discussion, $af(G_e,M')=af(G_e-w_1w_2,M')+1\leq Af(G)+1$.
      So $Af(G_e)=1+Af(G)$ and we are done.
   \end{proof}

   \begin{Rem}
      \label{RMM}
      For the perfect matching $M_e$ of $G_e$ as in Lemma \ref{Ev-Ad},
      $Af(G_e)=af(G_e,M_e)$.
   \end{Rem}

   Combining Lemmas \ref{Ev-NM}, \ref{ES} and \ref{Ev-Ad}, we can directly get the following result.

   \begin{Cor}
      \label{SS} Let $M$ be a perfect matching of a graph $G$ with $Af(G)=af(G,M)$, and $G_0$ be a quadrilateral subdivision of $G$ on $E(G)\setminus M$. Then $gf(G_0)=Af(G_0)$ if and only if $gf(G)=Af(G)$.
   \end{Cor}

\subsection{Ear Decomposition}

Each bipartite graph is given a black-white coloring on its vertices so that the end vertices of each edge receive different colors. An ear refers to a path of odd length.
A \emph{bipartite ear decomposition} of a bipartite graph $G$ can be expressed as below.
Let $e$ be an edge of $G$. Then join its end vertices by a path $P_1$ of odd length
(called ``first ear'').
Proceed inductively to build a sequence of bipartite graphs as follows:
if $G_{k-1}=e+P_1+P_2+\ldots +P_{k-1}$ has been constructed,
then join any two vertices in different
colors of $G_{k-1}$ by the $k$-th ear $P_k$ (of odd length) such that
$P_k$ has only its end vertices in
$G_{k-1}$. The decomposition $G=G_k=e+P_1+P_2+\ldots +P_k$
is called a bipartite ear decomposition
of $G$. For $1\leq i\leq k$, each $G_i$ is a matching covered conformal
subgraph of $G$.

\begin{The}
\cite{Lovasz1977}
   \label{ED}
   A bipartite graph is matching covered if and only
   if it has a bipartite ear decomposition.
\end{The}


A \emph{graded ear decomposition} of a matching covered graph $G$ is a sequence of matching covered subgraphs $K_2=G_0\subset G_1 \subset \ldots \subset G_m=G$
satisfying that for each $i\geq 1$,  $G_i$ is obtained from $G_{i-1}$ by attaching a set of disjoint ears.

\begin{The}
\cite{Lovasz1983}
   \label{GED}
   A matching covered graph $G$ has a graded ear decomposition starting with a matching covered conformal
   subgraph $G_0$:  $G_0\subset G_1 \subset \ldots \subset G_m=G$, so that $G_i =G_{i-1}+R_i$ is matching covered and
   $R_i$ contains at most two ears for each $1\leq i\leq m$.
\end{The}

The graded ear decomposition in Theorem \ref{GED} is called  \emph{nonrefinable}
if for any $R_i$  consisting of two ears $R_i'$ and $R_i''$, neither $G_{i-1}+R_i'$ nor $G_{i-1}+R_i''$ is matching covered.

\section{\label{Bip-C}Strongly uniform bipartite graphs}


In this section, we will characterize strongly uniform bipartite graphs by eliminating some special subgraphs.
First, we find four graphs $A_1,A_2,A_3$ and $A_4$ (see Fig. \ref{F01}) such that
$gf(A_i)=Af(A_i)+1$ (see Table \ref{AF}).   Let $\mathcal{A}:=\{A_i:1\leq i\leq 4\}$. We will see that a matchable bipartite graph is strongly uniform if and only if it has no conformal minor in $\mathcal{A}$. Meantime we can determine all types of such graphs. The key point is that we can show that if a matching covered bipartite  graph with at least four vertices has no  conformal minor from $\mathcal{A}$, then it possesses a Hamilton cycle. So  we need to consider possible positions of  all chords of the cycle.

\begin{figure}[H]
      \vspace{3pt}
      \centerline{\includegraphics[width=0.6\textwidth]{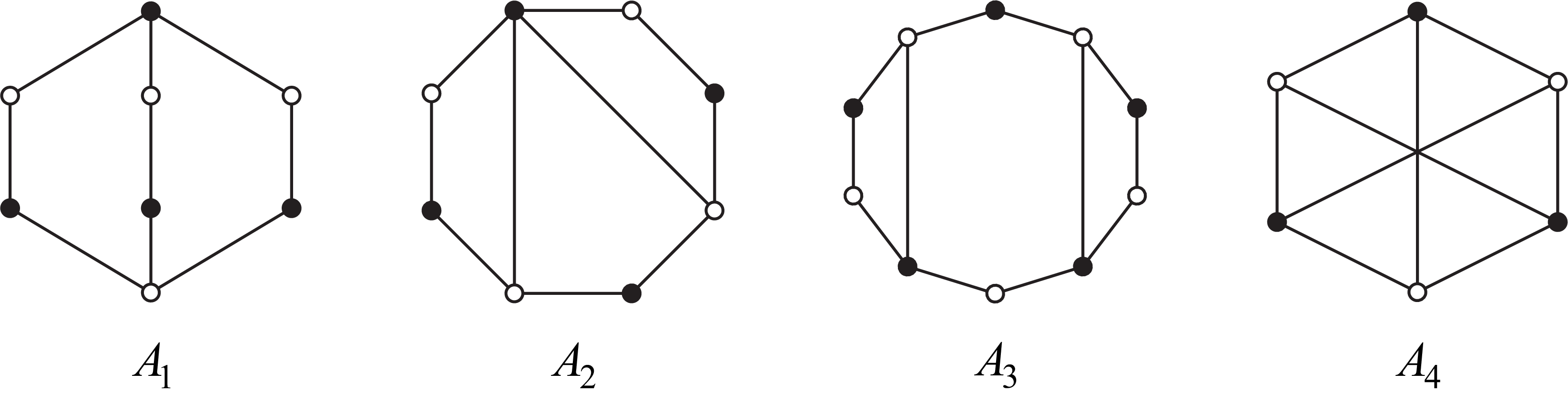}}
\caption{\label{F01}All graphs in $\mathcal{A}$.}
\end{figure}
\begin{table}[H]
   \centering
   \caption{\label{AF}The values of $gf(A_i)$ and $Af(A_i)$ for all graphs $A_i$ in $\mathcal{A}$.}
   \begin{tabular}{|c|c|c|c|}
      \hline
      & $A_1$ & $A_2,A_3$ & $A_4$ \\ \hline
      $gf(A)$ & 2 & 3 & 4 \\ \hline
      $Af(A)$ & 1 & 2 & 3 \\ \hline
   \end{tabular}
\end{table}

Given an even cycle $C$ with black-white coloring. A chord of $C$ is an edge not in $C$ which has end vertices on $C$. Let $x_1x_2$ and $y_1y_2$ be two disjoint chords of $C$ such that  $x_1$ and $y_1$ have the same color.
Then they are called \emph{parallel} if their end vertices are ordered as  $x_1,x_2,y_1,y_2$
along  $C$ (see $B_3$ in Fig. \ref{F02}), and \emph{crossed} if  their end vertices are ordered as  $x_1$, $y_1$, $x_2$, $y_2$
along  $C$ (see Fig. \ref{A1C2}).
Specially, two crossed chords $x_1x_2$ and $y_1y_2$ are said to be \emph{strongly crossed},
if $\{x_1y_2,x_2y_1\}\subseteq E(C)$ (see $B_2$ in Fig. \ref{F02}).

Let $\mathcal{B}_0$ be the set of all even cycles,
$\mathcal{B}_1$ be the set of bipartite graphs consisting of a Hamilton cycle and exactly one chord;
$\mathcal{B}_2$ be the set of graphs consisting of a Hamilton cycle and
exactly two strongly crossed chords
and $\mathcal{B}_3$ be the set of bipartite graphs consisting of a Hamilton cycle with pairwise parallel chords (see  Fig. \ref{F02}). We now have the following main result of this section.

\begin{figure}[H]
   \centerline{\includegraphics[width=0.6\textwidth]{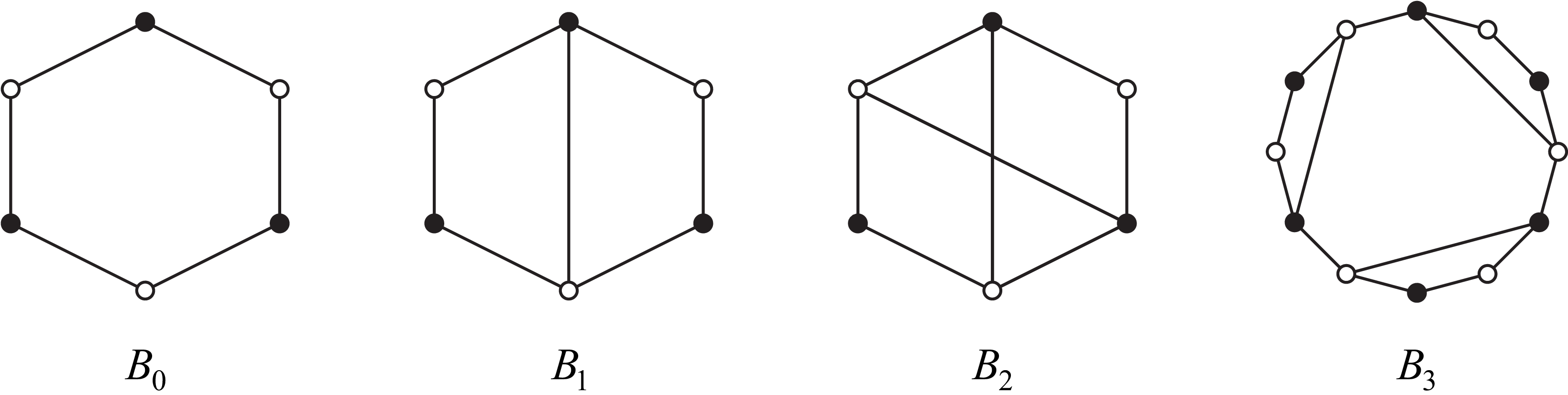}}
   \caption{\label{F02}The forms of the graphs in
   $\mathcal{B}_0\cup \mathcal{B}_1\cup \mathcal{B}_2\cup \mathcal{B}_3$.}
\end{figure}

   \begin{The}
      \label{BMT}
      Let $G$ be a matching covered bipartite graph with at least four vertices.
      Then the following three statements are equivalent.
      \begin{enumerate}[i)]
         \item $G$ is strongly uniform;
         \item $G$ has no conformal minor in $\mathcal{A}$;
         \item $G\in \cup_{i=0}^3 \mathcal B_i$.
      \end{enumerate}
   \end{The}
   \begin{proof}
      i) $\Rightarrow$ ii).
      Suppose to the contrary that $G$ contains a conformal minor in $\mathcal{A}$, say $A_i$.
      Then $G$ contains a conformal subgraph  $H$ which is a bisubdivision of $A_i$. By Lemmas \ref{ES-A} and \ref{ES} and Table \ref{AF}
      we have $gf(H)=gf(A_i)>Af(A_i)\geq Af(H)$, which implies that
       $G$ is not strongly uniform, a contradiction.

       ii) $\Rightarrow$ iii).
      By Theorem \ref{ED}, $G$ has a bipartite ear decomposition
      $G=e+P_1+P_2+\ldots +P_k$ with $k\geq 1$.

   \begin{Cla}
      \label{Ham}
      $G$ has a Hamilton cycle.
   \end{Cla}

   \begin{proof}
      We apply induction on $k$. For $k=1$, it is trivial.
      Next assume that $k\geq 2$ and
      the graph $G_{k-1}=e+P_1+P_2+\cdots +P_{k-1}$ has a Hamilton cycle $C$.
      If $P_k$ is an edge, then $C$ is also a Hamilton cycle of $G$.
      So we may assume that $P_k$ has odd length at least $3$.
      Let $u$ and $v$ be the end vertices of $P_k$.
      Then $uv\in E(C)$ for
      otherwise, $C+P_k$ is a bisubdivision of $A_1$, which implies that $A_1$ is a conformal minor of $G_k$, a contradiction.
      So $C+P_k-uv$ is a Hamilton cycle of $G_k$.
   \end{proof}

   Let $C=v_1v_2\ldots v_nv_1$ be a Hamilton cycle of $G$.
   Then we fix a plane drawing of $C$ and consider all chords.
   In the following, all chords are drawn in the interior of $C$, and
   all bisubdivisions we made are not on any chord.
   If there is at most one chord, then $G\in \mathcal{B}_0\cup \mathcal{B}_1$.
   Next we consider the case there are at least two different chords.
   Since $G$ has no conformal minor $A_2$,
   each pair of chords is disjoint.
   Since $G$ has no conformal minor $A_3$,
   every pair of chords is either crossed or parallel.

   \begin{Cla}
      \label{CP}
      Each pair of crossed chords of $C$ is strongly crossed.
   \end{Cla}

   \begin{proof}
      Suppose to the contrary that two crossed chords $x_1x_2$ and $y_1y_2$ of $C$ are not strongly crossed, where both $x_1$ and $y_1$ are black.
      Then we can assume that $x_1y_2\notin E(C)$.
      Let $P_0$ be a path from $y_1$ to $x_2$ along $C$;
      and let $G'$ be the graph obtained from $C+x_1x_2+y_1y_2$ by deleting the interior vertices of $P_0$ together with their incident edges.
      Then $G'$ consists of
      three internally disjoint  paths
      from $x_1$ to $y_2$ of odd length at least 3 (see Fig. \ref{A1C2}),
      which is isomorphic to a bisubdivision of $A_1$.
      Since $|V(P_0)|$ is even,
      $G'$ is a conformal subgraph of $G$. So $A_1$ is a conformal minor of $G$, a contradiction.
   \end{proof}
   \begin{figure}[H]
         \vspace{3pt}
         \centerline{\includegraphics[width=0.15\textwidth]{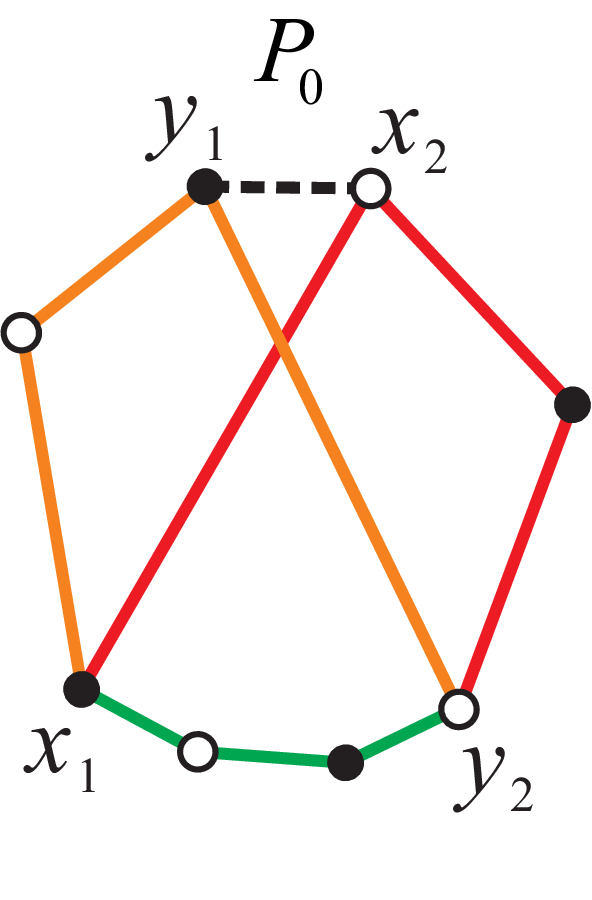}}
      \caption{\label{A1C2}Illustration for the proof of Claim \ref{CP}.}
   \end{figure}

   \begin{Cla}
      \label{cla2}
      There are no crossed pairs among three different chords (if exist).
   \end{Cla}

   \begin{proof}
      Suppose $C$ has three chords $x_1x_2$, $y_1y_2$ and $z_1z_2$,
      and $\{x_1x_2,y_1y_2\}$ is a crossed pair,
      where $x_1$, $y_1$ and $z_1$ are all black.
      By Claim \ref{CP}, $\{x_1y_2,x_2y_1\}\subseteq E(C)$.
      If $z_1z_2$ crosses both $x_1x_2$ and $y_1y_2$ (see Fig. \ref{F03}),
      then we get two new crossed pairs $\{x_1x_2,z_1z_2\}$ and $\{z_1z_2,y_1y_2\}$.
      Similarly, $\{x_1z_2,x_2z_1,z_1y_2,z_2y_1\}\subseteq E(C)$.
      Then $G=C+x_1x_2+y_1y_2+z_1z_2$ is isomorphic to $A_4$, a contradiction.
      So either $\{z_1z_2,x_1x_2\}$ or $\{z_1z_2,y_1y_2\}$
      is a parallel pair, say the former.
      Then $C+y_1y_2+z_1z_2$ is a bisubdivision of $A_3$, which implies that $A_3$ is a conformal minor of $G$, a contradiction.
   \end{proof}

      \begin{figure}[H]
            \vspace{3pt}
            \centerline{\includegraphics[width=0.35\textwidth]{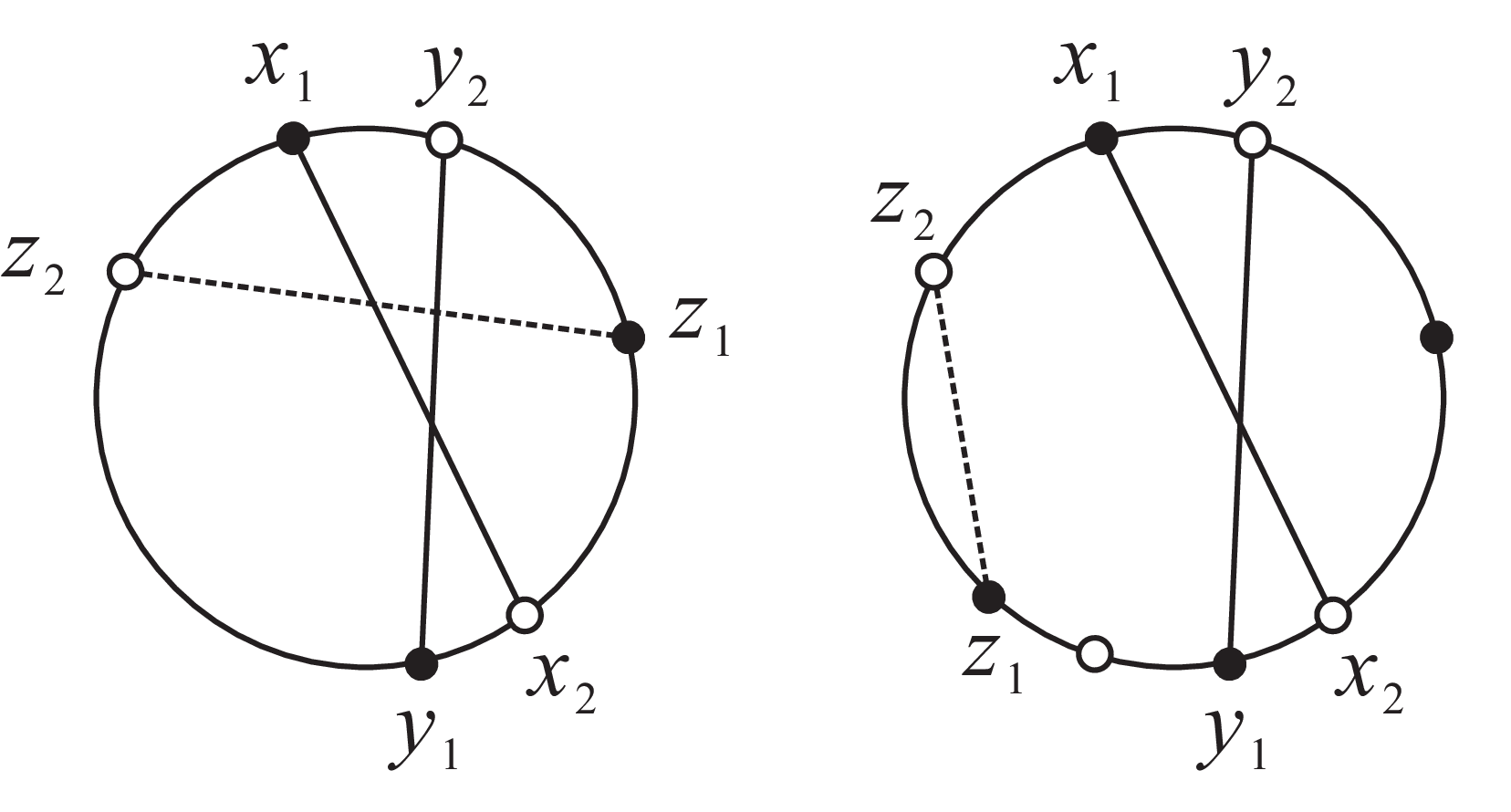}}
         \caption{\label{F03}Illustration for Claim \ref{cla2}.}
      \end{figure}
If $C$ has a pair of crossed chords, then Claims 2 and 3 imply that $G\in \mathcal B_2$ . Otherwise, $G\in \mathcal B_3$.     In conclusion, $G\in \cup_{i=0}^3\mathcal B_i$.

 iii) $\Rightarrow$ i). We can check that $G$ has no conformal minors in $\mathcal A$.
      Let $H$ be a conformal matchable subgraph of $G$. Let
      $H_1,H_2,\ldots,H_k$ be all elementary  components of  $H$, where $k\geq 1$.
      Then each $H_i$ is a matching covered bipartite subgraph without conformal minor in
      $\mathcal{A}$ for $1\leq i\leq k$, since a conformal minor of $H_i$ is also that of $G$.
      Since $gf(H)=\sum \limits_{i=1}^{k} gf(H_i)$ and $Af(H)=\sum \limits_{i=1}^{k} Af(H_i)$,
     it suffices to show that $gf(H_i)=Af(H_i)$.

      If $H_i$ is a $K_2$ (a complete graph with two vertices), then $gf(H_i)=Af(H_i)=0$. Otherwise, ii) $\Rightarrow$ iii) implies that $H_i\in \cup_{i=0}^3 \mathcal B_i$.
      If $H_i\in \mathcal B_0$, then $gf(H_i)=Af(H_i)=1$. If $H_i\in \mathcal B_1$ or $\mathcal B_3$ (see Fig. \ref{F02}), then
      it has a perfect matching containing all chords, say $M$.
      Further, we can find a compatible $M$-alternating set with size $c(H_i)$ where each pair of cycles  either is disjoint or shares only chords.
      So $Af(H_i)\geq af(H_i,M)\geq c(H_i)$.
      By Theorems \ref{Gf>Af} and \ref{gf(G)<c(G)},
      $c(H_i)\geq gf(H_i)\geq Af(H_i)\geq c(H_i)$, which implies that $gf(H_i)=Af(H_i)=c(H_i)$.
      The remaining case is  $H_i\in \mathcal B_2$. Then it has a perfect matching $M$ containing exactly one chord.
      Then we can also get $c(H_i)-1$ compatible $M$-alternating cycles.
      Since $H_i$ contains a nonconformal cycle, $c(H_i)-1\geq gf(H_i)$ by Theorem \ref{gf(G)<c(G)}.
      Once by Theorem \ref{Gf>Af}, $gf(H_i)=Af(H_i)=c(H_i)-1=2$.
      From above discussion, we can deduce that $gf(H)=Af(H)$ and $G$ is strongly uniform.
\end{proof}
We can see that  a matchable graph is strongly uniform if and only if each elementary component is  strongly uniform. So Theorem \ref{BMT} implies immediately the following corollary.

\begin{Cor}
   \label{gA}
   Let $G$ be a matchable bipartite graph.
   Then the following three statements are equivalent.
   \begin{enumerate}[i)]
      \item $G$ is strongly uniform;
      \item $G$ has no conformal minor in $\mathcal{A}$;
      \item Each elementary component of $G$  either is $K_2$  or belongs to  $\cup_{i=0}^3\mathcal B_i$.
   \end{enumerate}
\end{Cor}

\section{\label{GrEa} Strongly uniform BN-Graphs}



Like bipartite graphs, we further find 25 BN-graphs in Fig. \ref{F07}, each of which is nonbipartite and $gf(D_i)>Af(D_i)$ for each $1\leq i\leq 25$.
The verification is performed by a computer program  (see Table \ref{tab2}).

\begin{figure}[H]
   \begin{minipage}{1\linewidth}
      \vspace{3pt}
      \centerline{\includegraphics[width=1\textwidth]{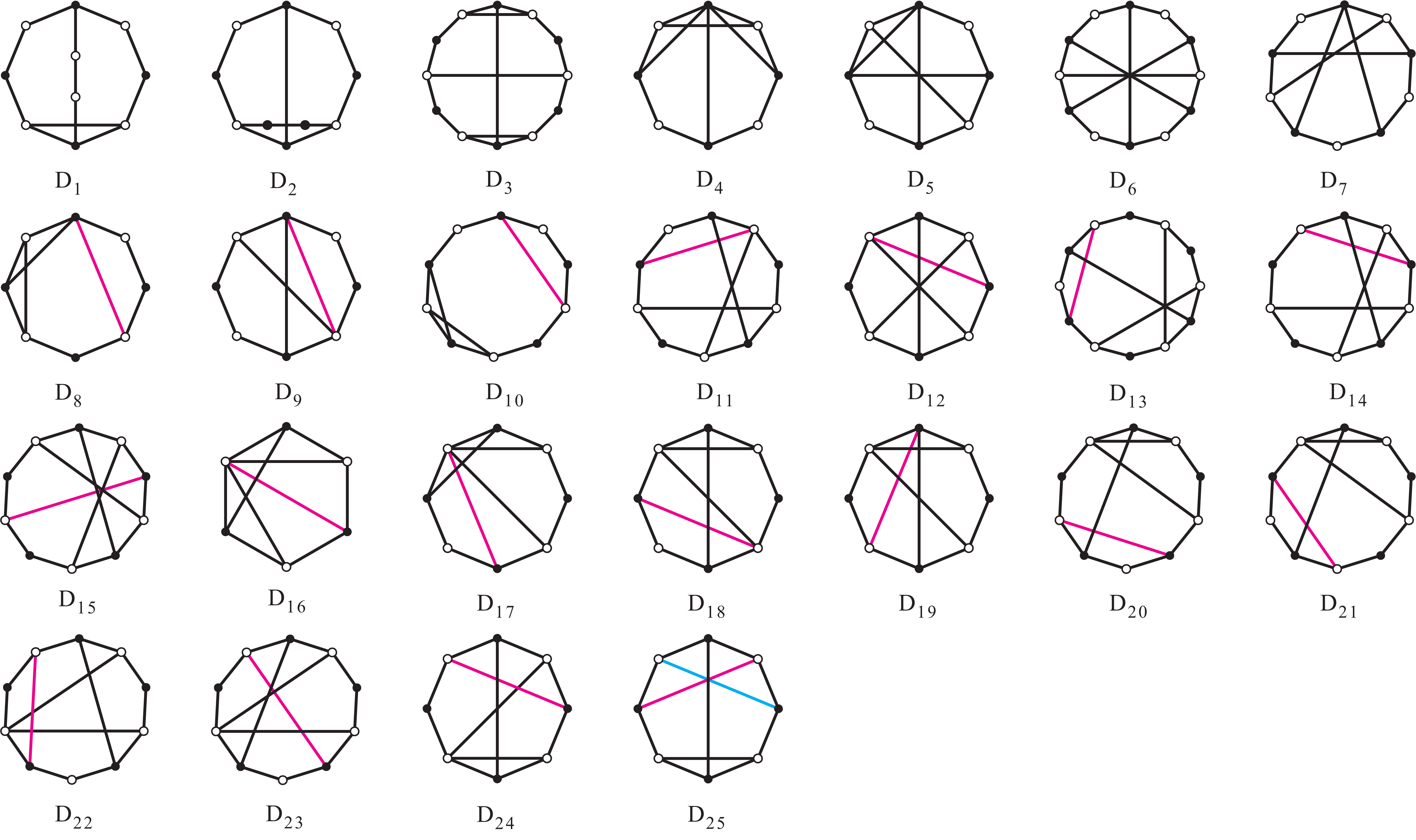}}
   \end{minipage}
   \caption{\label{F07} Excluded conformal minors $D_i$   for $1\leq i\leq 25$.}
\end{figure}

\begin{table}[ht]
   \centering
   \caption{\label{tab2}The global forcing and maximum anti-forcing numbers of all graphs in $\mathcal{D}$.}
   \begin{tabular}{|c|c|c|c|}
      \hline
      & $D_1,D_2$ & $D_3$ to $D_{24}$ & $D_{25}$\\ \hline
      $gf(D)$ & 2 & 3 & 4\\ \hline
      $Af(D)$ & 1 & 2 & 3\\ \hline
   \end{tabular}
\end{table}

By the definition,
each strongly uniform graph contains no conformal minors in $\mathcal{A}\cup \mathcal{D}$, where $\mathcal{D}$ denotes the set of graphs from $D_1$ to $D_{25}$.
Actually, the reverse is also true for BN-graphs. So we have the main result of this article as follows.

\begin{The}
   \label{odd-pre}
   Let $G$ be a matchable BN-graph.
   Then $G$ is strongly uniform if and only if $G$ has no conformal
   minor in $\mathcal{A}\cup \mathcal{D}$.
\end{The}
\subsection{Proof of Theorem \ref{odd-pre}}
We just need to prove the sufficiency of Theorem \ref{odd-pre}. Suppose that $G$ does not have any conformal minor in $\mathcal{A}\cup \mathcal{D}$.
Let $H$ be a conformal matchable subgraph of $G$. Let  $H_1,H_2,\ldots,H_k$ be  the elementary components of $H$, where $k\geq 1$, which are matching covered. Similar to the proof of Theorem \ref{BMT} it suffices to prove that
$gf(H_i)=Af(H_i)$. Since $H_i$ is a conformal subgraph of $G$, each $H_i$ has no conformal minors  in $\mathcal{A}\cup \mathcal{D}$. Further,  by Theorem \ref{EBV}, each $H_i$ is also a BN-graph. Hence each $H_i$ is
 a matching covered BN-subgraph without conformal minor in
$\mathcal{A}\cup \mathcal{D}$.
If $H_i$ is bipartite,  then $gf(H_i)=Af(H_i)$ by Theorem \ref{BMT}, so we are done.
Otherwise, we will get a ``structural'' characterization of $H_i$ by Lemma \ref{BNFG}
and obtain $gf(H_i)=Af(H_i)$ by Lemma \ref{G1234E} to complete the entire proof.
%
%
The definitions of graph families  $\mathcal{G}_i$ for $0\leq i\leq 3$
will be given in the next subsection, and the proofs of  Lemmas \ref{BNFG} and \ref{G1234E} appear in Sections 5 and 6 respectively.

\begin{Lem}
   \label{BNFG}
   Let $G$ be a nonbipartite matching covered BN-graph without conformal minors in $\mathcal{A}\cup \mathcal{D}$.
   Then $G\in \bigcup\limits_{i=0}^3\mathcal{G}_i$.
\end{Lem}

\begin{Lem}
   \label{G1234E}
   If $G\in \bigcup\limits_{i=0}^3\mathcal{G}_i$, then $gf(G)=Af(G)$.
\end{Lem}

Actually, the reverse of Lemma \ref{BNFG} also holds. But its proof is tedious, and we will not show it in this paper.


\subsection{Definitions of \texorpdfstring{$\mathcal{G}_i$ for $0\leq i\leq 3$}{Gi for 0 <= i <= 3}}
We find 26 BN-graphs shown in Figs. \ref{F09} to \ref{F12} that are all nonbipartite, and
verify by computer that these graphs have equal global forcing numbers and maximum anti-forcing numbers
(see Table \ref{tab3}). Note that all such graphs  are obtained from
an even cycle  by adding some  chords.
Next we   define $\mathcal{G}_i$ based on these graphs.

\begin{figure}[H]
   \vspace{3pt}
   \centerline{\includegraphics[width=0.75\textwidth]{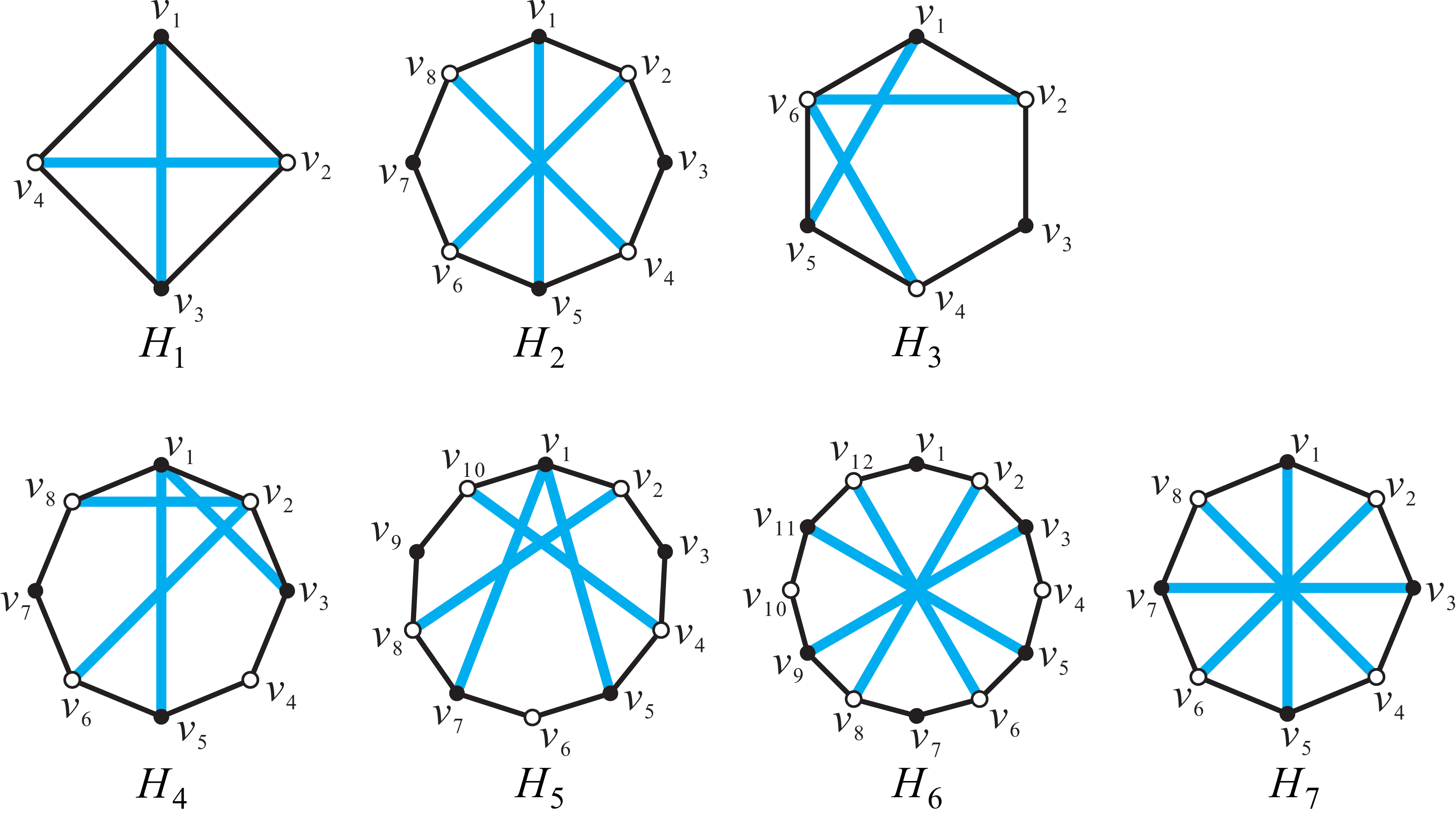}}
\caption{\label{F09} Fundamental graphs in $\mathcal{G}_0$.}
\end{figure}

\begin{figure}[ht]
   \vspace{3pt}
   \centerline{\includegraphics[width=1\textwidth]{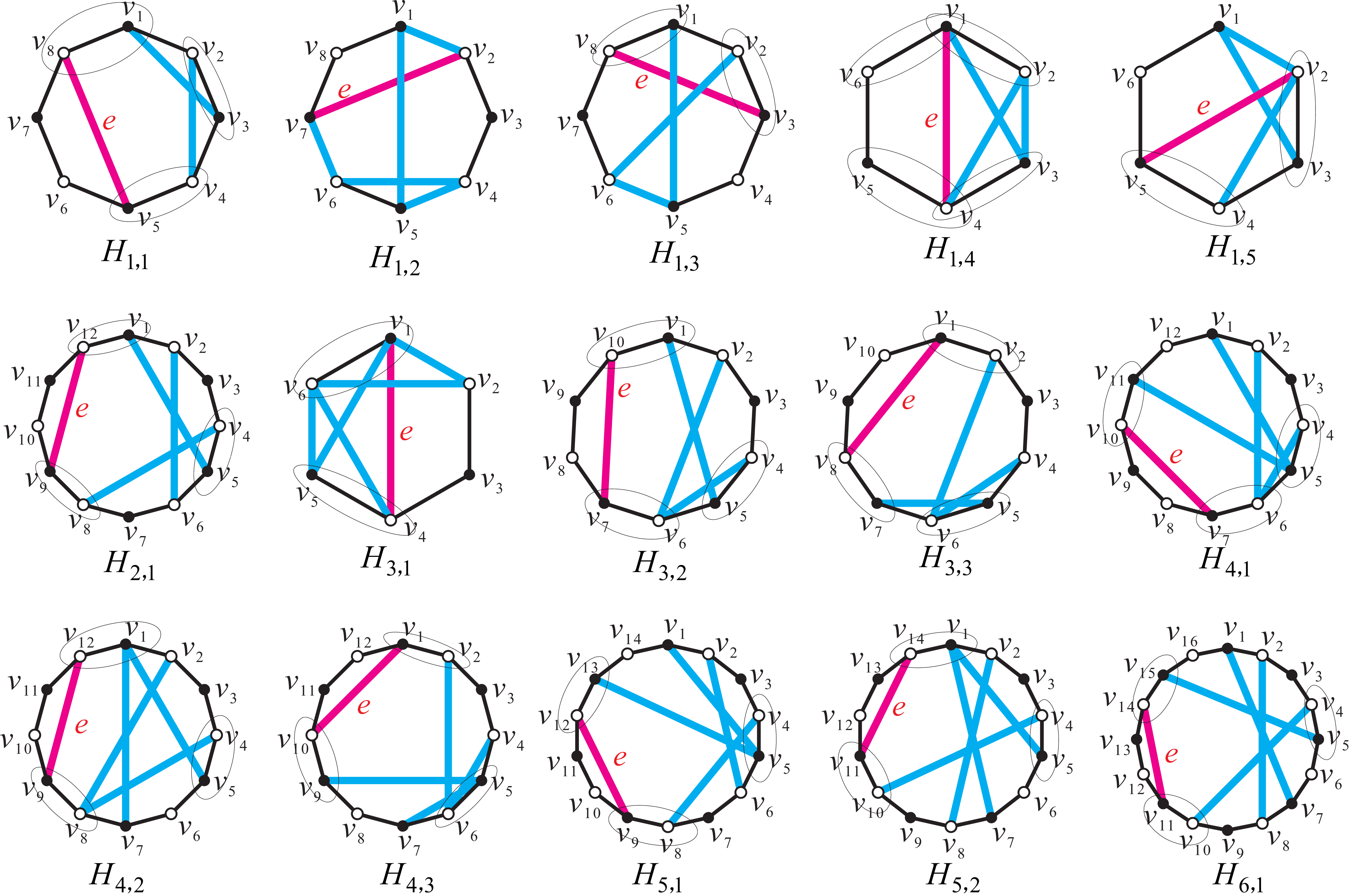}}
\caption{\label{F10} Fundamental graphs in $\mathcal{G}_1$: $H_{i,j}$ is obtained from a
bisubdivision of $H_i$ by adding a chord $e$ (red), where $1\leq i\leq 6$.}
\end{figure}

\begin{figure}[H]
   \vspace{3pt}
   \centerline{\includegraphics[width=0.4\textwidth]{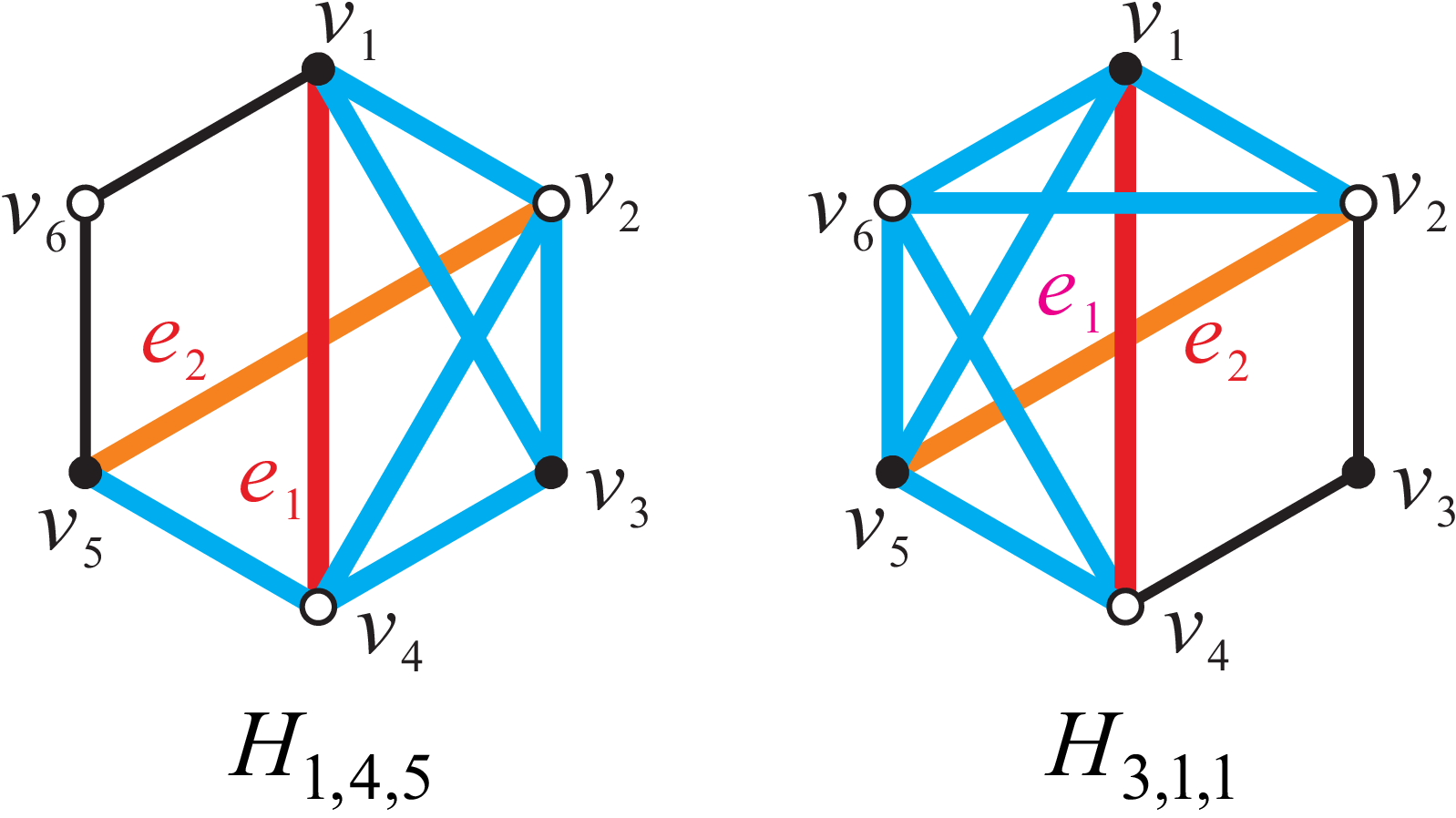}}
\caption{\label{F13} Fundamental graphs in $\mathcal{G}_2$: $H_{1,4,5}-e_2=H_{1,4}$,
$H_{1,4,5}-e_1=H_{1,5}$ and $H_{3,1,1}-e_i=H_{3,1}$ for $i=1,2$.}
\end{figure}

\begin{figure}[H]
   \vspace{3pt}
   \centerline{\includegraphics[width=0.4\textwidth]{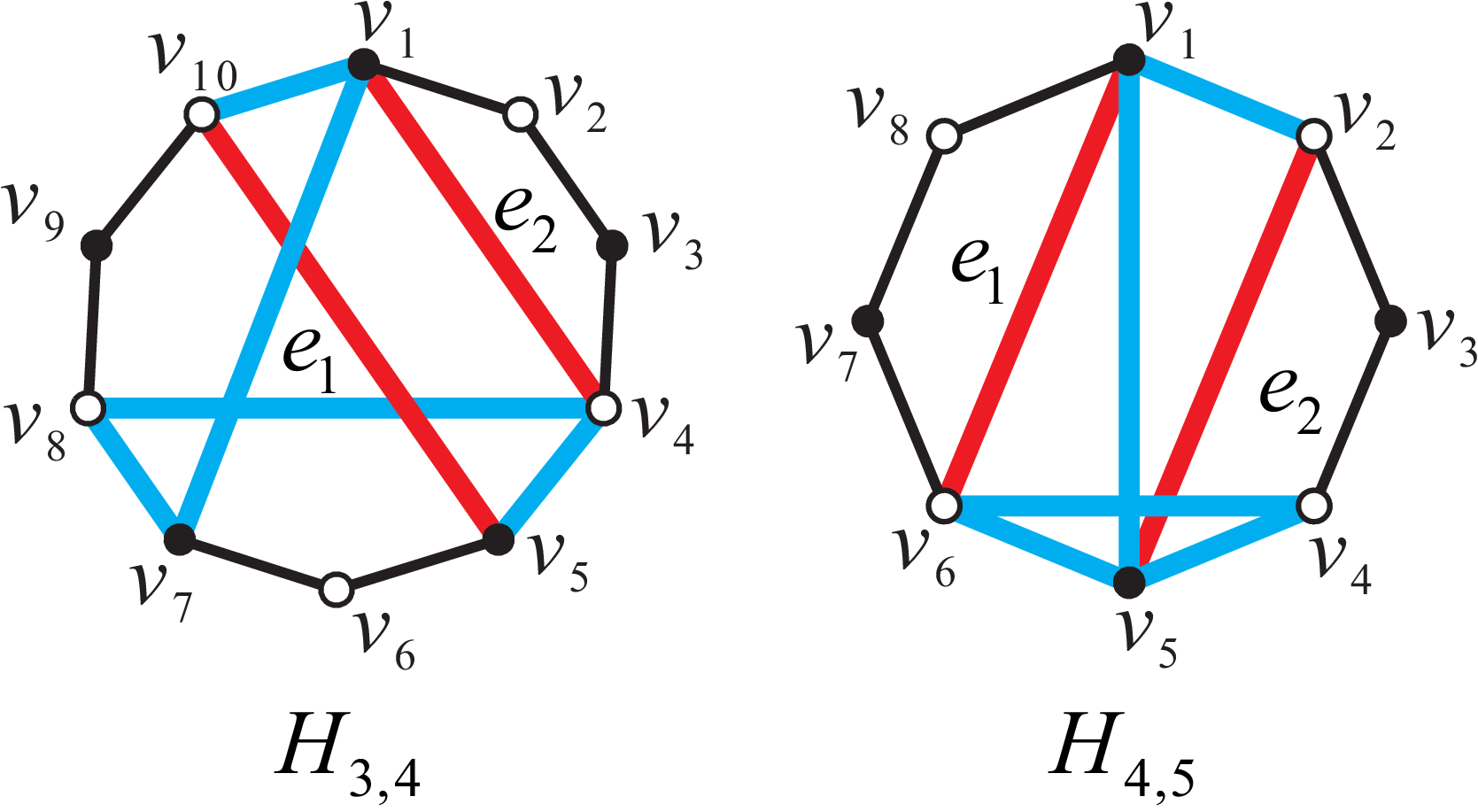}}
\caption{\label{F12}
$H_{3,4}-e_2$ and $H_{3,4}-e_1$ are bisubdivisions of $H_{1,3}$ and $H_{1,4}$, respectively. $H_{4,5}-e_2$ and $H_{4,5}-e_1$ are bisubdivisions of
$H_{1,4}$ and $H_{1,5}$, respectively.}
\end{figure}

\begin{table}[H]
   \centering
   \caption{\label{tab3} The values for $gf(H)$ and $Af(H)$ of all graphs $H$ in
   Figs. \ref{F09} to \ref{F12}.}
   \begin{tabular}{|c|c|c|c|}
      \hline
      & $H_1$ to $H_3$ & $H_4$ to $H_6$, $H_{1,1},H_{1,4},H_{2,1}$ to $H_{3,3}$ &
      $H_7$ \\ &  $H_{1,2}, H_{1,3},H_{1,5}$ & $H_{1,4,5}$, $H_{3,1,1}$, $H_{3,4}$, $H_{4,5}$ & $H_{4,1}$ to $H_{6,1}$\\
      \hline
      $gf(H)=Af(H)$ & 2 & 3 & 4\\ \hline
   \end{tabular}
\end{table}
In drawings of these graphs the thin edges (black) are called \emph{replaceable} and the others (bold or colored edges) are \emph{irreplaceable edges}. 
All replaceable edges of these fundamental graphs other than $H_{1,3}$ and $H_{1,5}$ form a \emph{replaceable set}.
For $H_{1,3}$, it has  seven thin edges, which do not form a replaceable set; in fact, $H_{1,3}$ has exactly two \emph{replaceable sets}:
$\{v_1v_2,v_3v_4,v_4v_5,v_6v_7,v_7v_8\}$ and $\{v_1v_8,v_2v_3,v_3v_4,v_4v_5,v_6v_7,v_7v_8\}$;  for  $H_{1,5}$, it also has exactly two \emph{replaceable sets}:
$\{v_2v_3,v_4v_5,v_5v_6,v_6v_1\}$ and $\{v_3v_4,v_5v_6,v_6v_1\}$.

Let $\mathcal{G}_0$, $\mathcal{G}_1$ and $\mathcal{G}_2$ be the set of graphs
obtained from a graph in Figs. \ref{F09}, \ref{F10} and \ref{F13} by some bisubdivisions on a \emph{replaceable set}, respectively.

The definition of $\mathcal{G}_3$ employs not only bisubdivisions, but also quadrilateral subdivisions.
For  $H_{1,4}$, two \emph{strong replaceable sets} are defined as $\{v_1v_2,v_3v_4\}$ and $\{v_1v_6,v_4v_5\}$.
For the other graphs in Fig. \ref{F10}, \emph{strong replaceable sets} are their edges indicated by ellipses.

Let $\mathcal{G}_3$ be the set of  graphs obtained from $H_{3,4}$ and $H_{4,5}$
by performing  bisubdivisions on their replaceable sets, together with 
 those  from each graph from  $H_{1,1}$ to $H_{6,1}$ except $H_{1,2}$ in Fig. \ref{F10} by performing
bisubdivisions on their replaceable sets and quadrilateral subdivisions on their strong replaceable sets.

\subsection{Operations on non-(strong) replaceable sets}

We adopt the following convention in Propositions \ref{pro1} and \ref{pro2} below: a bisubdivision of an irreplaceable set is defined to be either a bisubdivision of at least one irreplaceable edge in the outer cycle or a bisubdivision of two replaceable edges that lie in distinct replaceable sets such that each replacement-path has length at least three; and requiring a quadrilateral subdivision on an edge set $S$ such  that ``every'' replacement-path has length at least $5$ (See subsection 2.3, where  $k_i\ge 1$ for each edge $e_i$).


\begin{Pro}
   \label{pro1}
   For graphs  $H_{1,2}$ to $H_{1,5}$, $H_{3,1}$, $H_{1,4,5}$, $H_{3,1,1}$, $H_{3,4}$ or $H_{4,5}$,
   if we perform a bisubdivision on an irreplaceable set that lies entirely on the outer cycle, then the resulting graph contains a conformal minor $A_1$, $D_1$, $D_2$ or $D_9$.
\end{Pro}
\begin{proof}
   For  $H_{1,2}$, it has three irreplaceable edges $v_1v_2$, $v_4v_5$ and $v_6v_7$ on the outer cycle.
   When  performing a bisubdivision on an irreplaceable edge, we get a graph containing a conformal minor $A_1$, $D_1$ or $D_2$, which can be seen clearly from another drawing shown in Fig. \ref{E12B}.

   For  $H_{1,3}$, it has exactly one irreplaceable edge $v_5v_6$ on the outer cycle.
  If we  perform a bisubdivision on $v_5v_6$, then the resulting graph contains a conformal minor $A_1$; see  the second  drawing in Fig. \ref{E12B}.
   Further,  $H_{1,3}$  has the other irreplaceable sets $\{v_1v_2,v_1v_8\}$ and $\{v_1v_2,v_2v_3\}$. We   perform a bisubdivision on them to get a graph with a conformal minor $D_1$ (deleting $v_5v_6$); see  the third  drawing in Fig. \ref{E12B}.

   For  $H_{1,4}$, it has exactly one irreplaceable edge $v_2v_3$ on the outer cycle.
   If we  perform a bisubdivision on $v_2v_3$, then the resulting graph contains a conformal minor $D_9$.

   For  $H_{1,5}$, we  perform a bisubdivision on exactly one irreplaceable edge $v_1v_2$ to get a graph with a conformal minor $A_1$; see
   the fourth drawing in Fig. \ref{E12B}.
   We   perform a bisubdivision on irreplaceable sets $\{v_3v_4,v_2v_3\}$ or $\{v_3v_4,v_4v_5\}$,
   to get a graph with a conformal minor $D_1$ or $D_2$ (deleting $v_1v_2$); see
   the last drawing in Fig. \ref{E12B}.
\begin{figure}[H]
      \begin{minipage}{1\linewidth}
         \vspace{3pt}
         \centerline{\includegraphics[width=1\textwidth]{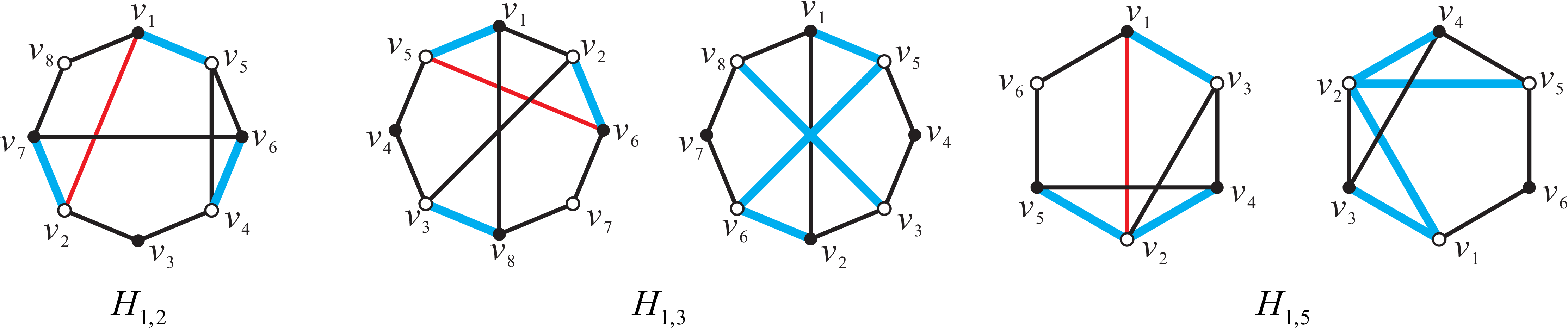}}
      \end{minipage}
      \caption{\label{E12B} New drawings of $H_{1,2}$, $H_{1,3}$ and $H_{1,5}$.}
   \end{figure}
   For  $H_{3,1}$, it has two irreplaceable edges $v_1v_2$ and $v_5v_6$ on the outer cycle.
   If we  perform a bisubdivision on $v_1v_2$ or $v_5v_6$,
   then the resulting graph contains a conformal minor $A_1$ or $D_9$.

   For  $H_{1,4,5}$, it has four irreplaceable edges $v_1v_2$, $v_2v_3$, $v_3v_4$ and $v_4v_5$ on the outer cycle.
   When  performing a bisubdivision on $v_1v_2$ or  $v_4v_5$, and $v_2v_3$ or $v_3v_4$, we get a graph with a conformal minor $A_1$ and $D_9$ respectively; see the left drawing in Fig. \ref{F11-2} for $v_3v_4$.

   For  $H_{3,1,1}$, it has four irreplaceable edges
   $v_1v_2$, $v_4v_5$, $v_5v_6$ and $v_1v_6$ on the outer cycle.
   If we   perform a bisubdivision on at least one of these edges, then the resulting graph contains a conformal minor $A_1$ or $D_9$.

   \begin{figure}[H]
      \begin{minipage}{1\linewidth}
         \vspace{3pt}
         \centerline{\includegraphics[width=0.5\textwidth]{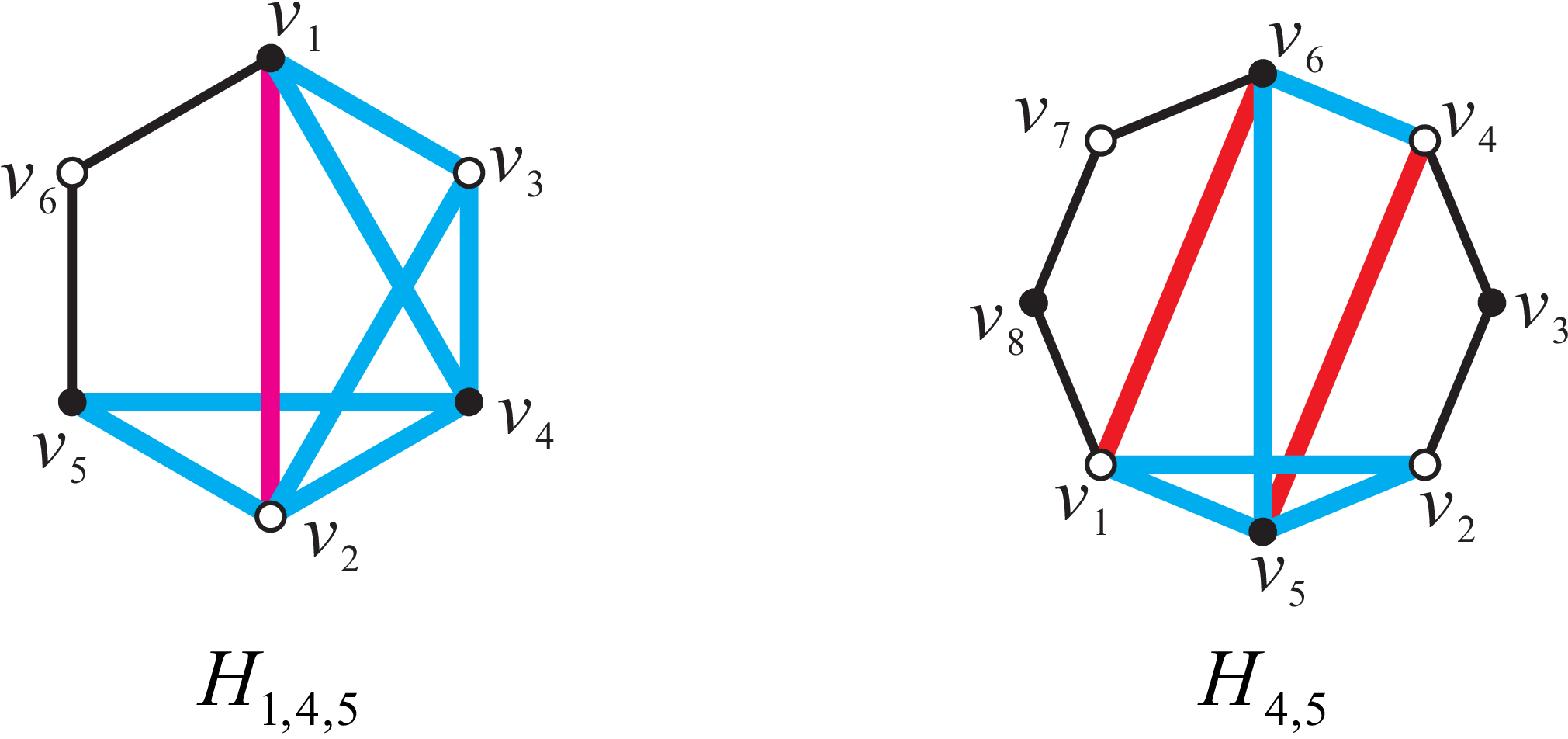}}
      \end{minipage}
      \caption{\label{F11-2} New drawings of $H_{1,4,5}$ and $H_{4,5}$.}
      \end{figure}

   For  $H_{3,4}$, we  perform a bisubdivision on irreplaceable edge $v_7v_8$,  $v_1v_{10}$ or $v_4v_5$
   to get a graph with a conformal minor $A_1$ and  $D_1$ respectively similar to the discussion on $H_{1,3}$.

   For  $H_{4,5}$, it has three irreplaceable edges $v_1v_2$, $v_4v_5$ and $v_5v_6$ on the outer cycle.
   If we  perform a bisubdivision on $v_4v_5$,
   then the resulting graph contains a conformal minor $D_9$.
   According to new drawings of $H_{4,5}$, if we perform bisubdivision on
   $v_1v_2$ or $v_5v_6$, then the resulting graph contains a conformal minor
   $D_2$ or $D_1$.
\end{proof}

\begin{Pro}
   \label{pro2}
   For each of the graphs  $H_{3,4}$, $H_{4,5}$ and $H_{1,1}$ to $H_{6,1}$ in Fig.~\ref{F10},
   let $S$ be a nonempty subset of a replaceable set such that $S$ is not contained in any strongly replaceable set. If  a quadrilateral subdivision is performed on $S$,
   then the resulting graph contains $A_3$ or $D_{10}$ a conformal minor.
\end{Pro}
\begin{proof}
   For $H_{3,4}$ and $H_{4,5}$, they have no strong replaceable sets.
   When   performing a quadrilateral subdivision on a replaceable edge, we can see that the resulting graph has a conformal minor $A_3$;
   For  $H_{1,2}$, it has no strong replaceable sets.
   If we   perform a quadrilateral subdivision on a replaceable edge, then the resulting graph contains a conformal minor $A_3$ or $D_{10}$; For $H_{1,4}$, it has exactly two strong replaceable sets
   $S_1=\{v_1v_2,v_3v_4\}$ and $S_2=\{v_1v_6,v_4v_5\}$.  Let $S$ be a replaceable set of $H_{1,4}$ that is not  a strongly replaceable set. If $S$ contains $v_5v_6$, then performing a quadrilateral subdivision on $S$  results in a graph containing a conformal minor $D_{10}$. Otherwise,  $S$ contains at least one edge from $S_1$ and at least one edge from $S_2$. In this case,
  performing a quadrilateral subdivision on $S$  results in a graph containing a conformal minor $A_3$;
   Similarly, for each graph from $H_{1,1}$ to $H_{6,1}$ except $H_{1,2}$ and $H_{1,4}$,
   performing a quadrilateral subdivision on a non-strong replaceable set yields a graph containing the  conformal minor $A_3$ or $D_{10}$.
\end{proof}

Props. \ref{pro1} and \ref{pro2} explain the reason why
we make bisubdivisions only on their replaceable sets and
quadrilateral subdivisions only on their strong replaceable sets for
all graphs in Figs. \ref{F09} to \ref{F12}. These restrictions are
necessary for the resulting graphs to contain no conformal minors in
$\mathcal{A}\cup\mathcal{D}$. 

In the following, we show that these restrictions are also sufficient.
When we say that $G$ is obtained from some graph $H$ in Figs. \ref{F09} to \ref{F12}, it means that $G$ is obtained from $H'$ by performing bisubdivisions on its replaceable set and/or quadrilateral subdivisions on its strong replaceable set by default, where $H'$ is either $H$ or is obtained from $H$ by adding some chords.

\subsection{Hamilton cycle}
Similarly to the proof  of Theorem \ref{BMT}, we first get a Hamilton cycle $C$ for each matching covered BN-graph without conformal minors in $\mathcal{A}\cup \mathcal{D}$.

\begin{Lem}
   \label{Ham-2}
   Let $G$ be a  matching covered BN-graph with at least four vertices and
   without conformal minors  $A_1,D_1$, or $D_2$. Then $G$ has a Hamilton cycle.
\end{Lem}

\begin{proof}
   According to Theorem \ref{GED}, $G$ has a nonrefinable graded ear decomposition $G=e+R_1+R_2+\ldots +R_k$
   such that each $R_i$ contains at most two ears, where $k\geq 1$.
   We will apply induction on $k$.
   For $k=1$, $R_1$ contains exactly one ear, since the above graded ear decomposition is nonrefinable.
   Then $G$ is an even cycle, and the result is obvious.
   Next we assume $k\geq 2$.
   By induction hypothesis, $G':=e+R_1+R_2+\ldots +R_{k-1}$ has a Hamilton cycle $C'$ with a proper coloring.
   If each ear in $R_k$ has length one, then $C'$ is also a Hamilton cycle of $G$.
   Next assume that $R_k$ contains an ear with length at least three, say $P$.

   \begin{Cas}
      \label{CRPD}
      $R_k$ contains an ear whose end vertices have different colors.
   \end{Cas}

   Let $P'$ be an ear in $R_k$ with end vertices $u$ and $v$,
   which have different colors.
   Then we can check that $G'+P'$ is matching covered, since $C'+P'$ is matching covered.
   By the definition of nonrefinable, $G_k=G'+P'$ and $P'=P$.
   If $uv\notin E(C')$, then $C'+P$ forms a bisubdivision of $A_1$, a contradiction.
   So $uv\in E(C')$ and $C=C'-uv+P$ is a Hamilton cycle of $G$.

   \begin{figure}[H]
         \vspace{3pt}
         \centerline{\includegraphics[width=0.45\textwidth]{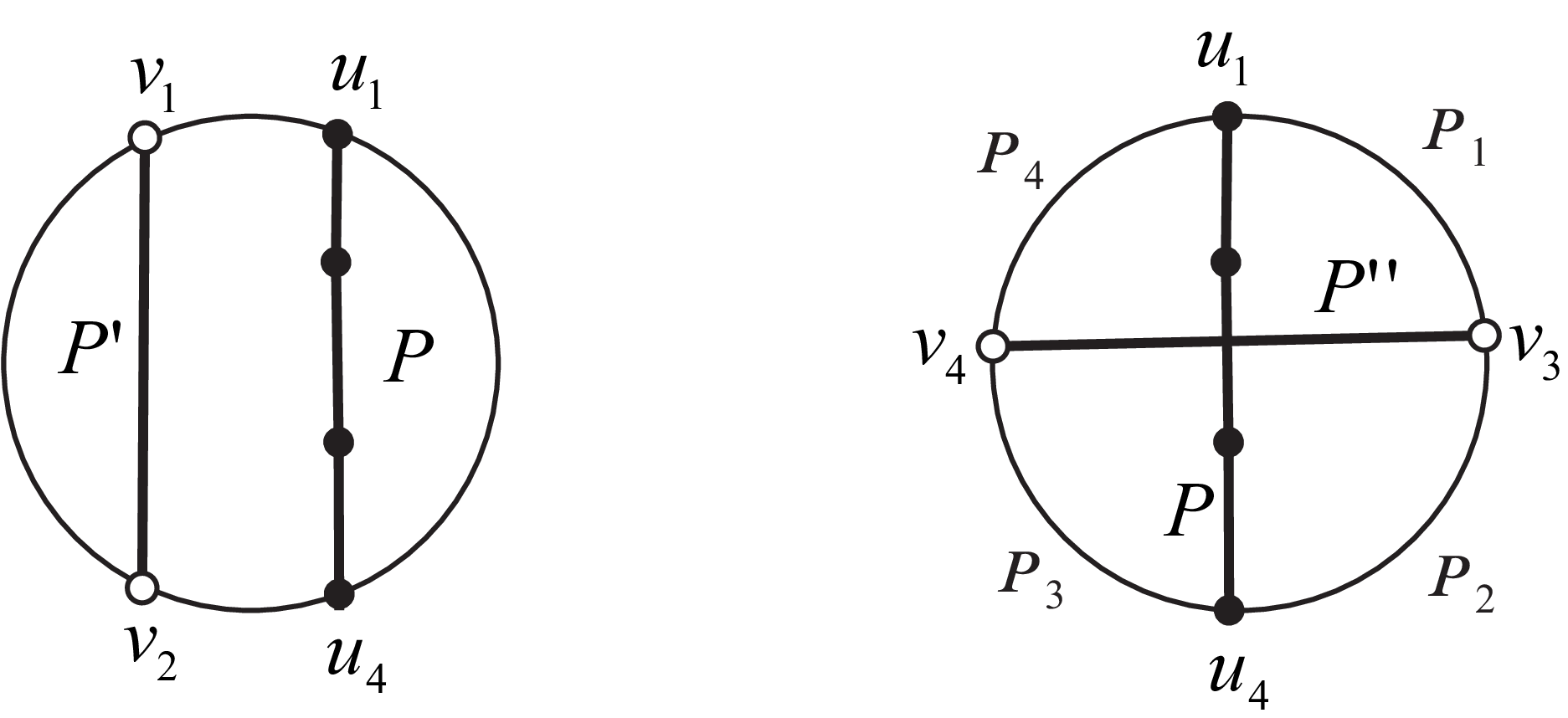}}
      \caption{\label{F08} Illustration for Case \ref{CRPS}.}
      \end{figure}
   \begin{Cas}
      \label{CRPS}
      Each ear in $R_k$ has end vertices with the same color.
   \end{Cas}

      Assume $P=u_1u_2u_3u_4$ and both $u_1$ and $u_4$ are black.
      Then $G$ has a perfect matching containing $\{u_1u_2,u_3u_4\}$.
      Let $P_u$ and $P'_u$ be the two paths from $u_1$ to $u_4$ along $C'$.
      Then their white vertices can not be completely matched by their black vertices.
      If $G$ contains a  path $P'$ joining two white vertices $v_1$ and $v_2$ on
      $P_u$ or $P'_u$,
      which is internally disjoint with $C'$ and $P$ (see Fig. \ref{F08}),
      then
      $C'+P+P'$ contains an odd conformal bicycle of $G$,
      a contradiction to Theorem \ref{EBV}.
      Otherwise, $G$ contains a path of odd length $P''$ joining one white vertex $v_3\in V(P_u)$
      and another white vertex $v_4\in V(P'_u)$,
      which is internally disjoint with $C'$ and $P$.
      %
      Then $C'$ is partitioned into four paths
      $P_1$, $P_2$, $P_3$ and $P_4$ in clockwise by end vertices of $P$ and $P''$.
      Assume that $u_1$, $v_3$, $u_4$ and $v_4$ appear on $C'$ in clockwise,
      and $P_1$ is a $u_1$-$v_3$ path.
      Since $G$ has no conformal minor in $\{D_1,D_2\}$,
      either $\{P_1,P_3\}$ or $\{P_2,P_4\}$ is an edge subset, say the former.
      Then $P+P_2+P''+P_4$ forms a Hamilton cycle of $G$ and we are done.
   \end{proof}

  \section{Proof of Lemma \ref{BNFG}}

 By Lemma \ref{Ham-2} $G$ has a  Hamilton cycle $C$. Also  we give a proper black-white coloring to the vertices of $C$. Since $G$ is nonbipartite,
   both end vertices of chords may have the same colors.
   Such chords are called \emph{monochromatic}, which are
    \emph{black} or \emph{white} according to color of the end vertices. Otherwise, a chord is
   \emph{bicolorable}.
   \begin{Lem}
   \label{SubgMC}
   Let $H$ be a subgraph of $G$ consisting of  $C$ together with some chords. If $H$   contains no monochromatic chords, or contains a pair of black and white chords, then $H$ is matching covered.
   \end{Lem}
   \begin{proof}
   Note that each edge in $E(C)$ belongs to some perfect matching of $H$.
   If $H$ does not contain monochromatic chords, then $H$ has a bipartite ear decomposition from $C$, so $H$ is matching covered by Theorem \ref{ED}.
   If $H$ contains a pair of black and white chords, then they are crossed. Otherwise  $H$ and thus $G$ contain an odd conformal bicycle, contradicting  Theorem \ref{EBV}.
  Take a white chord $u_1u_2$ and a black chord $v_1v_2$ of $C$ in $H$. Since $u_1u_2$ and $v_1v_2$ are crossed, $C-u_1-u_2-v_1-v_2$ consists of some disjoint paths of odd length, and thus has a perfect matching. Hence $u_1u_2$ and $v_1v_2$ belong to some perfect matching of $H$. So $H$ is matching covered.
   \end{proof}
   We always use $G'$ (resp. $G''$) to denote the subgraph of $G$ consisting of $C$ together with all monochromatic chords (resp. bicolorable chords).
   Then $G'$ and $G''$ are spanning subgraphs of $G$ without conformal minor in $\mathcal{A}\cup \mathcal{D}$. By Lemma \ref{SubgMC}, $G''$ is  a matching covered  bipartite graph. So $G''\in \bigcup \limits_{i=0}^3 \mathcal{B}_i$ by Theorem \ref{BMT}.
   In the following, we distinguish four cases according to $G''\in  \mathcal{B}_i, i=0,1,2,3,$ to  show that, if $G''\in \mathcal{B}_i$ then $G\in \mathcal{G}_i$ for
   $i=0,1,2$, and if $G''\in \mathcal{B}_3$, then $G\in \mathcal{G}_0\cup \mathcal{G}_1\cup \mathcal{G}_3$.

   Since $G$ is nonbipartite and matching covered, $G'$ contains at least one  black  chord and at least one white chord of $C$.
   Let $n_b$ and $n_w$ be the numbers of black and white chords in $G'$ respectively.
   Then $n_b\geq 1$ and $n_w\geq 1$.

\begin{Lem}
   \label{SCN}
    $G'\in \mathcal{G}_0$, and $G'$ is matching covered.
\end{Lem}
\begin{proof}
   As $G'$ has both black chord and white chords,
   $G'$ is matching covered by Lemma \ref{SubgMC}.
   We claim that any pair of white (or black) chords are either adjacent to or crossed with
   each other. Otherwise, suppose $C$ has a pair of disjoint white chords that are not crossed, which are crossed by a black chord. In this case $G'$ has a conformal subgraph consisting of three internally disjoint paths of odd length at least three between two vertices, which is a bisubdivision of $A_1$ (for example,  $F_1$ in Fig. \ref{PO} has three internally disjoint 3-length paths  from $v_2$ to $v_6$), a contradiction.

   \begin{figure}[H]
            \vspace{3pt}
            \centerline{\includegraphics[width=0.4\textwidth]{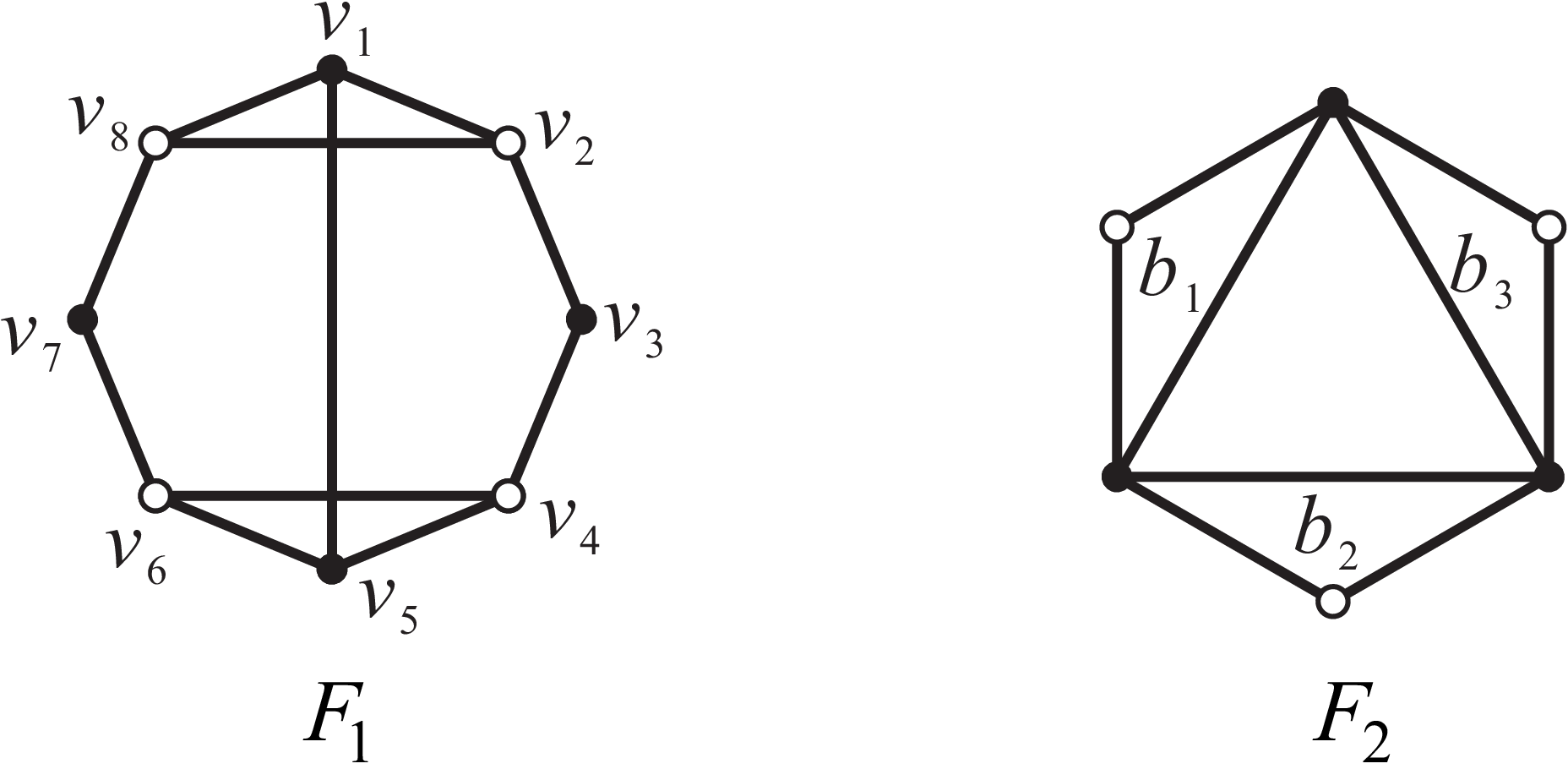}}
         \caption{\label{PO} Illustration for the proof of  Lemma \ref{SCN}.}
   \end{figure}

   Further we claim that $n_b\leq 2$ and $n_w\leq 2$. We only show the former. Suppose to the contrary that $C$ has three black chords $b_1,b_2$ and $b_3$. If the $b_i$'s have a common end-vertex, then $G'$ has a conformal minor  $D_4$, a contradiction.
   If the $b_i$'s form  a triangle  (see $F_2$ in Fig. \ref{PO}),
   it is impossible that a white chord  crosses each of black chord $b_1,b_2,b_3$.
   If the $b_i$'s form a 3-length path, then  $G'$ would have a conformal minor $D_5$, a contradiction.
   If the $b_i$'s form a matching, then $G'$ would have a conformal minor $D_6$, a contradiction. The remaining case is that among the $b_i$'s, there is  a pair of adjacent chords which is crossed with the third chord.
   In this case, $G'$ would have a conformal minor $D_7$, a contradiction.

   If $n_b+n_w=2$, then $n_b=n_w=1$ and $G'$ is a bisubdivision of $H_1$ on the thin edges.
   If $n_b+n_w=3$, then $G'$ is a  bisubdivision of $H_2$ or $H_3$.
   If $n_b+n_w=4$, then $n_b=n_w=2$ and $G'$ is a  bisubdivision of
   $H_4,H_5,H_6$ or $H_7$. In summary, $G'\in \mathcal{G}_0$.
%
   \end{proof}
\subsection{Case \texorpdfstring{$G''\in \mathcal{B}_0$}{G'' in B0}}

In this case, $G''=C$ and $G'=G$. By Lemma \ref{SCN}, we have $G\in \mathcal{G}_0$.

   \subsection{Case \texorpdfstring{$G''\in \mathcal{B}_1$}{G'' in B1}}
   \begin{Lem}
      \label{OBC}
      If  $G''\in \mathcal{B}_1$,
      then $G\in \mathcal{G}_1$.
   \end{Lem}
   \begin{proof}
      Since $G''\in \mathcal{B}_1$, $C$ has a unique bicolorable chord, denoted by  $e$.
      Next we mainly proceed by  excluding  bisubdivisions of $D_8$ to $D_{24}$ as conformal subgraphs. By Lemma \ref{SCN}, $G'\in \mathcal{G}_0$. Then $G'$ can be obtained from $H_1,H_2,\ldots,$ or $ H_7$ by a bisubdivision on edges of the outer cycle, and Hamilton cycle $C$ can be obtained accordingly from the outer cycle. So $G$ is obtained from $G'$ by adding the bicolorable chord $e$. Next we consider such seven cases.

      {\bf Subcase 1.}  $G'$ is obtained from   $H_1$.
      If $e$ is not adjacent to any monochromatic chord,  then $G$ can be obtained from  $H_{1,1}$, $H_{1,2}$ or $H_{1,3}$ (via a bisubdivision on replaceable sets) by using Prop. \ref{pro1} and excluding conformal minor $D_{10}$.

      If  $e$ is adjacent to exactly one monochromatic chord, then
      $e$ must cross another monochromatic chord by the exclusion of $D_8$, so
      $G$ can be obtained from $H_{1,5}$.
      Otherwise, $e$ is adjacent to two monochromatic chords,
      and $G$ can be obtained from $H_{1,4}$.

      {\bf Subcase 2.}  $G'$ is obtained from   $H_2$.
      In this case, we can show that $G$ can be obtained from $H_{2,1}$  (via a bisubdivision on replaceable sets).   By excluding conformal minors  $D_8$ to $D_{15}$, we have  seven possibilities $F_3$ to $F_8$ or $H_{2,1}$ of adding $e$ to $G'$ to get $G$ (see Fig. \ref{OFS}). It suffices to exclude the first six possibilities.
      Since $F_3-v_2v_6$ and $F_4-v_2v_6$ can be obtained from $H_{1,5}$ by a
      bisubdivision on an irreplaceable edge  (see the path $v_5v_6v_7v_8$ in $F_3$ and
      the path $v_1v_2v_3v_4$ in $F_4$),
      $G$ cannot be obtained from $F_3$ or $F_4$ by Prop. \ref{pro1}.
      Similarly,  $F_5-v_2v_6$ and $F_6-v_2v_6$ can be obtained from $H_{1,2}$ by a
      bisubdivision on an irreplaceable edge (see the vertices $v_6$ and $v_7$ in $F_5$
      and  $v_2$ and $v_3$ in $F_6$),
      and $F_7-v_2v_6$ can be obtained from $H_{1,3}$ by a bisubdivision on an irreplaceable edge (see the path $v_1v_2v_3v_4$),
      $G$ cannot be obtained from $F_5,F_6$ or $F_7$.
      Since $F_8$ contains three internally
      disjoint paths of odd length from $v_4$ to $v_{10}$: $v_4v_5v_1v_{10}$, $v_4v_8v_9v_{10}$ and
      $v_4v_3v_2v_6v_7v_{10}$,
      which form a conformal minor $A_1$,
      $G$ cannot be obtained from $F_8$.

      \begin{figure}[ht]
            \vspace{3pt}
            \centerline{\includegraphics[width=0.95\textwidth]{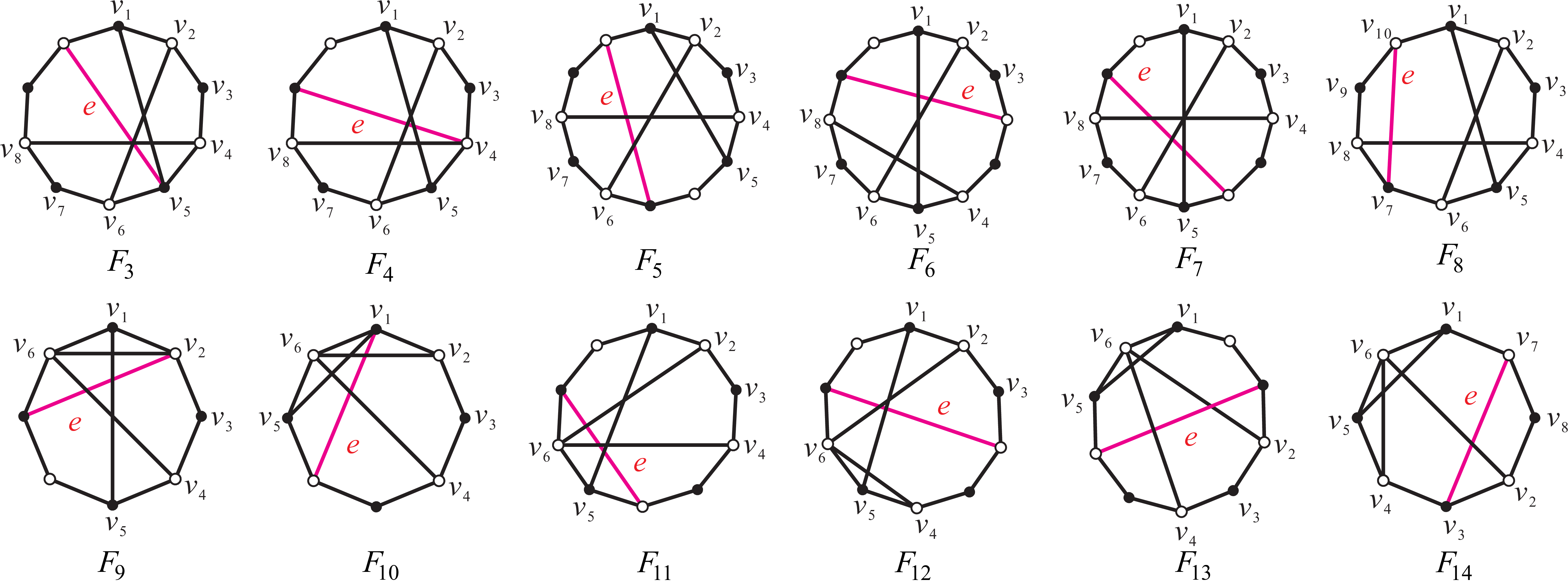}}
         \caption{\label{OFS} Illustration for the proof of Lemma \ref{OBC}.}
      \end{figure}

       {\bf Subcase 3.}  $G'$ is obtained from   $H_3$. 
      Similarly we can show that $G$ can be obtained from $H_{3,1}, H_{3,2}$ or $H_{3,3}$. Since $F_9-v_4v_6$ and $F_{10}-v_2v_6$ can be obtained from
      $H_{1,5}$ by performing bisubdivisions on its irreplaceable edges (see the path
      $v_2v_3v_4v_5$ in $F_9$ and the path $v_1v_2v_3v_4$ in $F_{10}$),
      $G$ has no conformal minor in $\{F_9,F_{10}\}$.
      From the exclusions of $D_8$ and $D_{16}$ to $D_{19}$,
      we can deduce that $e$ is adjacent to zero or two monochromatic chords.
      For the latter, $e$ must cross a monochromatic chord by excluding $D_8$ or $D_9$, so
      $G$ can be obtained from  $H_{3,1}$.
      For the former, $G$ cannot be obtained from   $F_{11}$ to $F_{13}$,
      since $F_{11}-v_2v_6$ can be obtained from $H_{1,3}$ by a bisubdivision on
       irreplaceable edge $v_5v_6$
      (corresponding to the path $v_1v_2v_3v_4$ in $F_{11}$), and
      $F_{12}-v_2v_6$ and $F_{13}-v_2v_6$ can be obtained from $H_{1,2}$ by a bisubdivision on
      irreplaceable edge $v_1v_2$ (see the vertices $v_2$ and $v_3$ in $F_{12}$ and $F_{13}$).
      By the exclusions of $D_{22}$ to $D_{24}$,
      $e$ crosses at most one monochromatic chord.
      Since $F_{14}$ contains three internally disjoint paths of odd length
      from $v_6$ to $v_7$: $v_6v_2v_8v_7$, $v_6v_4v_3v_7$ and $v_6v_5v_1v_7$,
      which form a conformal minor $A_1$,
      $G$ cannot be obtained from $F_{14}$.
      Once by the exclusions of $D_{10}$ and $D_{20}$ to $D_{21}$,
      $G$ can be obtained from $H_{3,2}$ or $H_{3,3}$.

       {\bf Subcase 4.}  $G'$ is obtained from $H_4$.
      Let $e_i$, $1\leq i\leq 4$, be the monochromatic chords of $C$. Since $G-e_i$ is also a matching covered BN-graph without conformal minors in $\mathcal A\cup \mathcal D$ by Lemma \ref{SubgMC}, $G-e_i$ can be obtained from  $H_{3,1}, H_{3,2}$ or $H_{3,3}$ by Subcase 3. We claim that $e$ is not adjacent to any $e_i$. If $e$ is adjacent to $e_1$, choose  $e_2$ and $e_3$ that are not adjacent to $e$. Then $G-e_4$ cannot be obtained from $H_{3,j}$, $j=1,2,3$, a contradiction.  Similarly we can show that $e$  does not cross any $e_i$. In this way we can deduce that
      $G$ can be obtained from $H_{4,j}$ for $j=1,2,3$.

       {\bf Subcase 5.}  $G'$ is obtained from   $H_5$ or $H_6$.  Take a monochromatic chord $e'$ so that the other monochromatic chords are disjoint. Similarly to Subcase 4,  $G-e'$  can be obtained from  $H_{2,1}$ by Subcase 2.  By possible position of $e'$ that does not  cross $e$ and the exclusion of $D_{10}$,  $G$
 can be obtained from  $H_{5,1}$, $H_{5,2}$ or $H_{6,1}$.

       {\bf Subcase 6.}  $G'$ is obtained from   $H_7$.
  Similar to Subcase 4,  the deletion of each monochromatic chord from $G$ results in a bisubdivision of $H_{2,1}$ from Subcase 2, which implies that $e$ does not cross any monochromatic chord. So $G$ contains a conformal minor  $D_{10}$,  a contradiction.

      In summary, $G$ can be obtained from  $H_{1,1}$ to $H_{6,1}$, so $G\in \mathcal{G}_1$.
   \end{proof}
   \subsection{Case \texorpdfstring{$G''\in \mathcal{B}_2$}{G'' in B2}}
   \begin{Lem}
      \label{G1B3}
      If  $G''\in \mathcal{B}_2$,
      then $G\in \mathcal{G}_2$.
   \end{Lem}
   \begin{proof}
      Let $e_1$ and $e_2$ be the two strongly crossed bicolorable chords of $C$ in $G$. Then $G-e_i$ is still a matching covered BN-graph by Lemma \ref{SubgMC}. Since $G-e_i$, $i=1,2$, also has no conformal minors in $\mathcal A\cup \mathcal D$, by Lemma \ref{OBC}, $G-e_i\in \mathcal G_1$. Then $G$ can be obtained from a graph $Q\in \mathcal{G}_1$ by adding a bicolorable chord.
      Since the exclusions of $D_8$ and $D_{10}$, $Q$ can not be obtained from $H_{1,1}$. Similarly, $Q$ can be obtained from $H_{1,j}, 2\leq j\leq 5$, or  $H_{3,1}$.

       If $Q$ is obtained from $H_{1,2}$, then $G$ is formed by combining $H_{1,2}$ and  $H_{1,2}$ or $H_{1,4}$,
         since  $e_1$ and $e_2$ are strongly crossed bicolorable chords of $C$.
         For the former $G$ can be obtained from  $D_{25}$, a contradiction. For the latter  $G$ can be obtained from $F_{15}$, and the deletion of $v_6v_7$ results in a bisubdivision  of $D_8$  (see Fig. \ref{F11}), a contradiction.

         \begin{figure}[H]
               \vspace{3pt}
            \centerline{\includegraphics[width=0.8\textwidth]{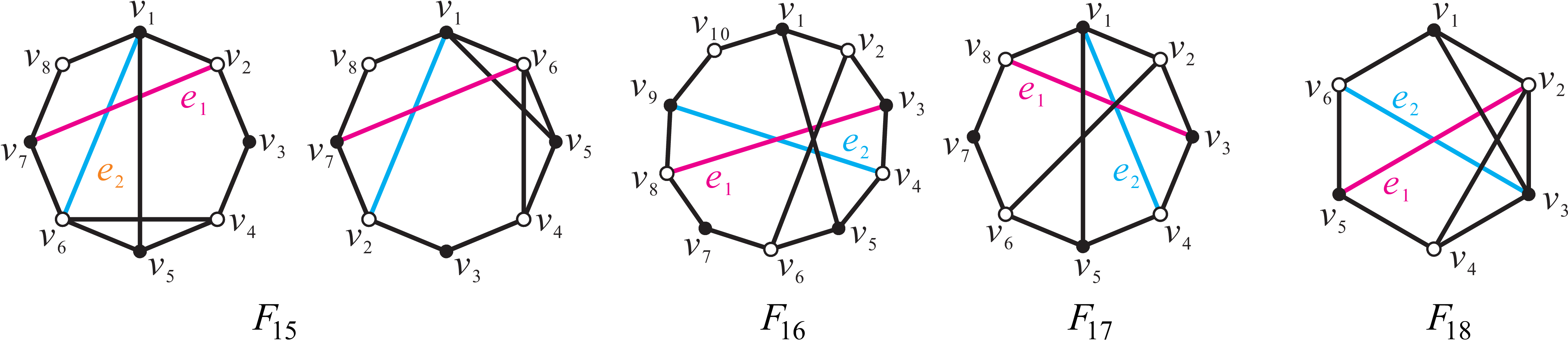}}
            \caption{\label{F11} Illustration for the proof of Lemma \ref{G1B3}.}
         \end{figure}

         If $Q$ is obtained from $H_{1,3}$, then similarly
         $G$ can be obtained from either $F_{16}$ or $F_{17}$,
         which contains  an odd conformal bicycle either ($v_2v_3v_8v_7v_6v_2$,
         $v_4v_5v_1v_{10}v_9v_4$)  or $(v_1v_4v_5v_1, v_2v_3v_8v_7v_6v_2)$ (see Fig. \ref{F11}), a contradiction to Theorem \ref{EBV}.

      If $Q$ is obtained from $H_{1,4}$, then $G$ can be obtained by  combining $H_{1,4}$  with $H_{1,5}$, which produces the graph $H_{1,4,5}$ (see Fig. \ref{F13}).
      If $G-e_i, i=1,2,$ is always obtained from   $H_{1,5}$, then $G$ can be obtained from $F_{18}$ (see Fig. \ref{F11}), which contains an odd conformal bicycle $(v_1v_3v_6v_1, v_2v_4v_5v_2)$, a contradiction.
      If $G'$ can be obtained from $H_3$, then
      $G-e_i, i=1,2,$ is always   a bisubdivision of $H_{3,1}$,
      and $G$ can be obtained from $H_{3,1,1}$ (see Fig. \ref{F13}).
      According to the definitions of replaceable sets of $H_{1,4,5}$ and $H_{3,1,1}$ and
      Prop. \ref{pro1},  $G\in \mathcal{G}_2$.
   \end{proof}
   \subsection{Case \texorpdfstring{$G''\in \mathcal{B}_3$}{G'' in B3}}
   \begin{Lem}
      \label{CS}
      If  $G''\in \mathcal{B}_3$,
      then $G\in \mathcal{G}_0\cup \mathcal{G}_1\cup \mathcal{G}_3$.
   \end{Lem}
   \begin{proof}
      By Lemma \ref{SubgMC}, the deletion of bicolorable chords from $G$ results in a matching covered BN-graph without conformal minor in $\mathcal A\cup \mathcal D$.
      Since $G''\in \mathcal{B}_3$,
      $G$ can be obtained from a graph in $\mathcal{G}_1$ by adding some bicolorable chords so that all  bicolorable chords are  parallel to each other.
      Let $n_c$ be the number of  bicolorable chords which are crossed with or adjacent to a monochromatic chord.
      If $n_c\leq 1$, then $G$ can be obtained from a graph ($H_{1,1}$ to $H_{6,1}$, except $H_{1,2}$) shown in Fig. \ref{F10}
      by performing bisubdivisions on its replaceable set and
      quadrilateral subdivisions on its strong replaceable set by Props. \ref{pro1} and \ref{pro2}. So $G\in \mathcal G_3$.

      Next we  assume that $n_c\geq 2$.
      Let $e$ and $e'$ be such two  bicolorable chords  crossed with or adjacent to a monochromatic chord.
      Similarly we can see that  both $G'+e'$ and $G'+e$ can be obtained from $H_{3,1}$ or $H_{1,j}, 2\leq j\leq 5$.
      Note that $H_{3,1}$ is obtained from $H_{1,4}$ by adding a monochromatic chord. Thus we know that both $G'+e'$ and $G'+e$ can be obtained from $H_{1,j}, 2\leq j\leq 5$.\\[7pt]
      \noindent\textbf{Claim 1.}
         Neither $G'+e'$ nor $G'+e$ can be obtained from  $H_{1,2}$.
      \begin{proof}
         Suppose $G'+e$ is obtained from $H_{1,2}$. By considering all irreplaceable edges $v_1v_2$, $v_4v_5$ and $v_6v_7$ of $H_{1,2}$,  we can deduce that $G'+e'$ can be obtained from $H_{1,3}$, $H_{1,4}$ or $H_{1,5}$.

         If $e'$  crosses two monochromatic chords of $H_{1,2}$,
         then both end vertices of $e'$ belong to two paths, corresponding to  $v_2v_3v_4$ and $v_5v_6$ in $H_{1,2}$, respectively.
         By considering the irreplaceable edge $v_5v_6$ of $H_{1,3}$ and parallel chords $e$ and $e'$,  $G'+e'\notin \mathcal{G}_1$, a contradiction.
         If $e'$ is adjacent to two monochromatic chords of $H_{1,2}$, then it must join $v_5$ and $v_6$ since $e$ and $e'$ are parallel and  $v_4v_5$ is an irreplaceable edge of $H_{1,2}$.
         By considering the irreplaceable edge $v_2v_3$ of $H_{1,4}$,
         $G'+e\notin \mathcal{G}_1$, a contradiction.
         If $e'$ is adjacent to a monochromatic chord and crossed with a monochromatic chord of $H_{1,2}$, then $e$ joins $v_6$ and a black vertex on the path corresponding to $v_2v_3v_4$, since $e$ and $e'$ are parallel and $v_6v_7$ is an irreplaceable edge of $H_{1,2}$. In this case, $H_{1,2}+e'$ contain three internally disjoint 3-length paths from  $v_1$ to $v_6$: $v_1v_2v_3v_6$, $v_1v_5v_4v_6$ and $v_1v_8v_7v_6$, which form a conformal minor $A_1$, a contradiction.
      \end{proof}

      \noindent\textbf{Claim 2.}
        If $G'+e$ is obtained from $H_{1,3}$,
         then $G'+e'+e$ can be obtained from $H_5$, $H_6$, $H_7$ or $H_{3,4}$.
      \begin{proof}
         If $e'$ crosses two monochromatic chords of $H_{1,3}$,
         then $G'+e'+e$ can be obtained from $H_{3,3,0}$ or $H_{3,3,1}$ , which is isomorphic to $H_7$ or $H_6$ (see the left two columns of Fig. \ref{F12c}).
         If $e'$ is adjacent to one monochromatic chord and crossed with a monochromatic chord of $H_{1,3}$, then $e'$ is incident with either $v_1$ or $v_2$ in $H_{1,3}$ by    $G''\in \mathcal{B}_3$, say $v_1$.
         Then the other end vertex of $e'$ belongs to the path corresponding to the edge $v_2v_3$ of $H_{1,3}$. Further $G'+e'+e$ can be obtained from $H_{3,5,0}$, which is isomorphic to $H_5$.
         If $e'$ is adjacent to two monochromatic chords of $H_{1,3}$, then it must join $v_1$ and $v_2$, since $v_5v_6$ is an irreplaceable edge of $H_{1,3}$. It implies that $G'+e'+e$ can be obtained from $H_{3,4}$ (see Fig. \ref{F12}).
      \end{proof}

   \begin{figure}[H]
         \centerline{\includegraphics[width=0.75\textwidth]{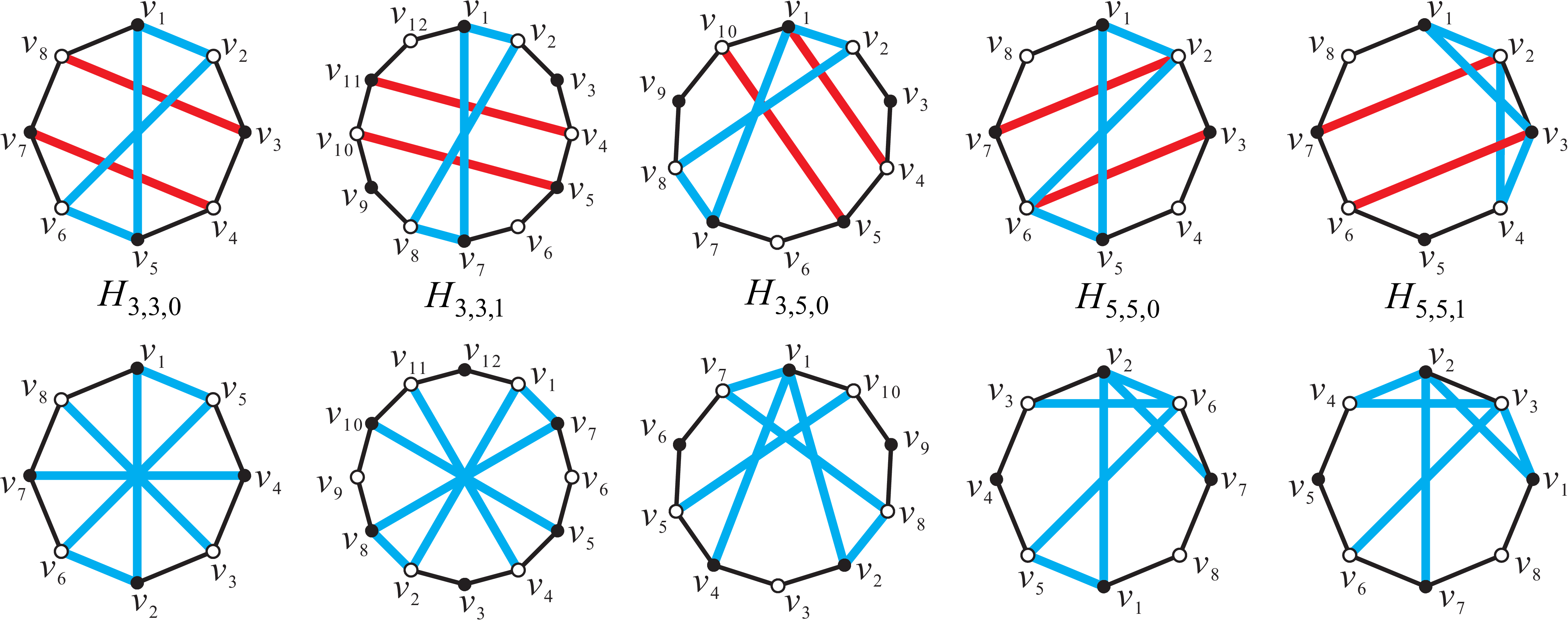}}
      \caption{\label{F12c} Illustration for Claims 2 and 3.} 
   \end{figure}

     \noindent\textbf{Claim 3.}
         If  neither $G'+e'$ nor $G'+e$ is obtained from $H_{1,3}$, then $G'+e'+e$ can be obtained from $H_4$ or $H_{4,5}$.

      \begin{proof}
         If $G'+e$ is obtained from $H_{1,4}$,
         then $G'+e'$ can not be obtained from $H_{1,4}$ since
         $v_2v_3$ is  irreplaceable  and $e$ and $e'$ are parallel to each other.
         By Claim 1,
         $e'$ is adjacent to one monochromatic chord and crossed with another monochromatic chord of $H_{1,4}$.
         So $e'$ is incident with either $v_2$ or $v_3$, say $v_3$, and with a vertex in the path corresponding to the edge $v_1v_2$ in $H_{1,4}$, which implies that $G'+e+e'$ can be obtained from  $H_{4,5}$ (see Fig. \ref{F12}).

         If both $G'+e$ and $G'+e'$ are obtained from  $H_{1,5}$,  then $e'$ is adjacent to one monochromatic chord and crossed with a monochromatic chord of $H_{1,5}$.
         So $e'$ is incident with either $v_3$ or $v_4$ in $H_{1,5}$.
         If $v_3\in V(e')$, then the other end vertex of $e'$ belongs to the path corresponding to the edge $v_4v_5$ in $H_{1,5}$,  which produces $H_{5,5,1}$ (see Fig. \ref{F12c}).
         If $v_4\in V(e')$, then the other end vertex of $e'$ belongs to the path corresponding to the edge $v_2v_3$ in $H_{1,5}$, which produces $H_{5,5,0}$.
         Since both $H_{5,5,0}$ and $H_{5,5,1}$ are isomorphic to $H_4$, $G'+e'+e$ can be obtained from $H_4$.
      \end{proof}

      By Claims 1, 2 and 3,
      $G'+e'+e$ can be obtained from $H_4$, $H_5$, $H_6$, $H_7$, $H_{3,4}$ or $H_{4,5}$.
      If  $G'+e'+e$ can be obtained from $H_{3,4}$ or $H_{4,5}$, then $G=G'+e+e'\in \mathcal G_3$ since we can not add any bicolorable chord to $H_{3,4}$ or $H_{4,5}$ by considering their irreplaceable edges.
      So we may suppose that  $G'+e'+e$ is obtained from  $H_4$, $H_5$, $H_6$ or $H_7$. In this case $G^*:=G'+e+e'$ has a new Hamilton cycle $C^*$ with only monochromatic chords. If $G=G^*$, then $G\in \mathcal G_0$. Otherwise,  by Lemma \ref{SCN} more edges of $G$ than $G^*$ are bicolorable chords of $C^*$, which are
      added  to $H_4$ to $H_7$.  By Lemma \ref{OBC}, each such chord can not be crossed with or adjacent to any monochromatic chord, since the resulting graph can be obtained from a graph from $H_{4,1}$ to $H_{6,1}$ in Fig. \ref{F10}. Since all such chords together with $e$ and $e'$ are pairwise parallel relative to $C$,  more such chords are added to $H_{4,1}$ to $H_{6,1}$ by quadrilateral subdivisions on their strong replaceable sets. For $H_{3,3,1}$, more chords are added on bisubdivisions of edges $v_{10}v_{11}$ and $v_4v_5$. For $H_{3,3,0}$, $H_{3,5,0}$, $H_{5,5,0}$ and $H_{5,5,1}$, we can add more chords by similar method.
      So $G$ can be obtained from $H_{4,1}$ to $H_{6,1}$ by bisubdivisions on their replaceable sets and/or by quadrilateral subdivisions on their strong replaceable sets, and
       $G\in \mathcal{G}_1\cup \mathcal{G}_3$.
   \end{proof}
   \section{Proof of Lemma \ref{G1234E}}
   As mentioned in Subsection 4.2 by a computer program we have confirmed that each fundamental graph $H$ in Figs. \ref{F09} to \ref{F12}  satisfies $gf(H)=Af(H)$
   (see Table \ref{tab3}). The further proof is divided into the following two lemmas.

\begin{Lem}
   \label{G123}
   If $G\in \mathcal{G}_0\cup \mathcal{G}_1\cup \mathcal{G}_2$, then $gf(G)=Af(G)$.
\end{Lem}
\begin{proof}
   We need to show that bisubdivisions on replaceable sets have no effect on the equality.
   By Lemma \ref{ES}, a bisubdivision does not change the global forcing number of a graph.
   So we only consider the change of  maximum anti-forcing number after
   a bisubdivision. We consider three cases as follows.

 {\bf Case 1.} $G\in \mathcal{G}_0$.
If $G$ is obtained from $H_1$,
   let $M$ be the perfect matching of $G$ containing
   crossed white and black chords.
   Then $G$ contains two compatible $M$-alternating cycles,
   corresponding to $v_1v_3v_2v_4v_1$ and $v_1v_3v_4v_2v_1$ in $H_1$,
   which share only the two monochromatic chords.
   By Lemma \ref{ES-A}, $2=Af(H_1)\geq Af(G)\geq af(G,M)\geq 2$,
   which implies that $gf(G)=Af(G)=2$.
   If $G$ is obtained from $H_2$ or $H_3$,
   similarly we have $gf(G)=Af(G)=2$, since it also has crossed white and black chords.

   If $G$ is obtained from $H_4$,
    let $M$ be the perfect matching of $G$ containing two crossed monochromatic chords
   $v_1v_3$ and $v_2v_8$.
   Then $G$ has three compatible $M$-alternating cycles corresponding to
   $v_1v_3v_2v_8v_1$, $v_2v_6v_7v_8v_2$ and
   $v_1v_3v_4v_5v_1$ in $H_4$, which share only monochromatic
   chords in $M$. Similarly, we obtain $gf(G)=Af(G)=3$.

   If $G$ is obtained from $H_5$,
   let $M$ be the perfect matching of $G$ containing two crossed monochromatic chords
   $v_1v_7$ and $v_2v_8$.
   Then $G$ contains three compatible $M$-alternating cycles corresponding to
   $v_1v_7v_8v_2v_1$, $v_1v_7v_6v_5v_1$ and
   $v_2v_3v_4v_{10}v_9v_8v_2$ in $H_5$, which share only two monochromatic
   chords in $M$. So $gf(G)=Af(G)=3$.

   If $G$ is obtained from $H_6$,
   let $M$ be the perfect matching of $G$ containing two crossed monochromatic chords
   $v_5v_{11}$ and $v_6v_{12}$.
   Then $G$ contains three compatible $M$-alternating cycles corresponding to
   $v_6v_7v_8v_2v_1v_{12}v_6$, $v_3v_4v_5v_{11}v_{10}v_9v_3$ and $v_5v_6v_{12}v_{11}v_5$ in $H_6$, which share only $v_5v_{11}$ and $v_6v_{12}$. So $gf(G)=Af(G)=3$.

   If $G$ is obtained from $H_7$,
   let $M$ be the perfect matching of $G$ containing all monochromatic chords.
   Then $G$ contains four compatible $M$-alternating cycles corresponding to $v_1v_2v_6v_5v_1$,
   $v_2v_3v_7v_6v_2$, $v_3v_4v_8v_7v_3$ and $v_4v_5v_1v_8v_4$ in $H_7$,
   which share only monochromatic chords. So $gf(G)=Af(G)=4$.

{\bf Case 2}. $G\in  \mathcal{G}_1$.
If $G$ is obtained from $H_{1,2}$, $H_{1,3}$ or $H_{1,5}$,
then let $M$ be the perfect matching of $G$ containing the unique bicolorable chord. It follows that $G$ contains two compatible $M$-alternating cycles and thus $gf(G)=Af(G)=2$.

   If $G$ is  obtained from $H_{1,1}$, let $M$ be the perfect matching of $G$ containing all chords. Then $G$ contains three compatible $M$-alternating cycles and thus $gf(G)=Af(G)=3$.

   If $G$ is obtained from $H_{1,4}$, let $M$ be the perfect matching of $G$ containing the unique bicolorable chord. Since $v_2v_3$ is an irreplaceable edge of $H_{1,4}$, $v_2v_3\in M$. Then $G$ contains three compatible $M$-alternating cycles and thus $gf(G)=Af(G)=3$.

If $G$ is obtained from $H_{2,1}$, $H_{3,2}$ or $H_{3,3}$,
then let $M$ be a perfect matching of $G$ containing a pair of crossed black and white chords and the unique bicolorable chord; if $G$ is obtained from $H_{3,1}$, then let $M$ be a perfect matching of $G$ containing chords $\{v_1v_5,v_4v_6\}$ in $H_{3,1}$. We can check that $G$ contains three compatible $M$-alternating cycles and thus $gf(G)=Af(G)=3$.

If $G$ is obtained from $H_{4,1}$ to $H_{4,3}$, then let $M$ be a perfect matching of $G$ containing chords
$\{v_4v_6,v_5v_{11},v_7v_{10}\}$ in $H_{4,1}$,
$\{v_1v_7,v_2v_8,v_9v_{12}\}$ in $H_{4,2}$, or
$\{v_4v_6,v_5v_7,v_1v_{10}\}$ in $H_{4,3}$.
If $G$ is obtained from $H_{5,1}$ to $H_{6,1}$, then let $M$ be a perfect matching of $G$ containing chords
$\{v_1v_5,v_2v_6,v_9v_{12}\}$ in $H_{5,1}$,
$\{v_1v_5,v_4v_{10},v_{11}v_{14}\}$ in $H_{5,2}$ or
$\{v_1v_7,v_2v_8,v_{11}v_{14}\}$ in $H_{6,1}$.
Obviously $G$ contains four compatible $M$-alternating cycles and thus $gf(G)=Af(G)=4$.

{\bf Case 3.} $G\in \mathcal{G}_2$. Then $G$ is obtained from $H_{1,4,5}$ or $H_{3,1,1}$. Let $M$ be a perfect matching of $G$ containing two crossed monochromatic chords. Then $G$ contains three compatible $M$-alternating cycles. It follows that $gf(G)=Af(G)=3$.
\end{proof}

\begin{Lem}
   \label{G24}
   If $G\in \mathcal{G}_3$, then $gf(G)=Af(G)$.
\end{Lem}
\begin{proof}
   If $G$ is  obtained from $H_{3,4}$ or $H_{4,5}$ (see Fig. \ref{F12}),
   then $gf(G)=gf(H_{3,4})=gf(H_{4,5})=3$ by Lemma \ref{ES} and Table \ref{tab3}.
   Let $M$ be the perfect matching of $G$ containing the two bicolorable chords.
   Then $G$ contains three compatible $M$-alternating cycles,
   which implies that $Af(G)\geq 3$.
   By Theorem \ref{Gf>Af}, $gf(G)=Af(G)=3$.

   If $G$ is obtained from a graph in $H_{1,1}$ to $H_{6,1}$ in Fig. \ref{F10}, other than $H_{1,2}$, by performing
   bisubdivision on their replaceable sets and quadrilateral subdivisions on its strong replaceable sets, then each strong replaceable set of a graph in $H_{1,1}$ to $H_{3,3}$, other than $H_{1,2}$, does not intersect the corresponding perfect matching attaining maximum anti-forcing number defined in Lemma \ref{G123}. As performing a bisubdivision on a replaceable set of $H_{1,1}$ to $H_{6,1}$ other than $H_{1,2}$ results in a graph in $\mathcal{G}_1$. By Case 2 in Lemma \ref{G123}, each graph in $\mathcal{G}_1$ has equal global and maximum anti-forcing number. By Corollary \ref{SS}, $gf(G)=Af(G)$.
\end{proof}


\begin{thebibliography}{99}\setlength{\itemsep}{0mm}\linespread{1.2}\selectfont
   \bibitem{Balas1981}
   E. Balas, Integer and fractional matchings, in: Studies on Graphs and Discrete Programming, P. Hansen, ed., North Holland Math. Stud. 59, 1981, pp. 1-13.
   \bibitem{Birkhoff1976}
   G. Birkhoff, Three observations on linear algebra, Univ. Nac. Tacum\'an Rev. Ser. A, 5 (1946) 147-151.
   \bibitem{cai2012global}
   J. Cai, H. Zhang, Global forcing number of some chemical graphs, MATCH Commun. Math. Comput. Chem. 67
   (2012) 289-312.
   \bibitem{CLM2005}
   M.H. de Carvalho, C.L. Lucchesi, U.S.R. Murty, On the number of dissimilar pfaffian orientations of graphs, Theor. Inform. Appl. 39 (2005) 93-113.
   \bibitem{Car}
   M.H. de Carvalho, C.L. Lucchesi, U.S.R. Murty, The perfect matching polytope and solid bricks, J. Comb. Theory Ser. B 92 (2004) 319-324.
   \bibitem{Chen20201}
   X. Chen, G. Ding, W. Zang, Q. Zhao, Ranking tournaments with no errors I: Structural description, J. Combin. Theory Ser. B 141 (2020) 264-294.
   \bibitem{CCLSV2005}
   M. Chudnovsky, G. Cornu\'ejols, X. Liu, P. Seymour, K.  Vu\v skovi\'c, Recognizing Berge graphs, Combinatorica 25 (2005) 143-186.
 \bibitem{CG}S.J. Cyvin, I. Gutman, Kekul\'e Structures in Benzenoid Hydrocarbons,  Springer-Verlag, Berlin, Heidelberg,  1988.
   \bibitem{D.T2007Global}
   T. Do\v sli\'c, Global forcing number of benzenoid graphs, J. Math. Chem. 41 (2007) 217-229.
   \bibitem{Kai2017Anti}
   K. Deng, H. Zhang, Anti-forcing spectra of perfect matchings of graphs, J. Comb. Optim. 33 (2017) 660-680.
   \bibitem{edmonds1965maximum}
   J. Edmonds, Maximum matching and a polyhedron with (0,1) vertices, J. Res. Nat. Bur. Stand. B 69 (1965) 125-130.
   \bibitem{EK} L. Esperet, F. Kardo\v{s}, A.D. King, D. Kr\'al, S. Norine, Exponentially many perfect matchings in
cubic graphs, Adv. Math. 227 (2011) 1646-1664.
   \bibitem{Fisher1961}
   M.E. Fisher, Statistical mechanics of dimers on a plane lattice, Phys. Rev. 124 (1961) 1664-1672.
   \bibitem{GT2011}
   B. Guenin, R. Thomas, Packing directed circuits exactly,  Combinatorica 31 (2011) 397-421.
   \bibitem{FL2001}
   I. Fischer, C.H.C. Little, A characterization of Pfaffian near-bipartite graphs, J. Combin. Theory Ser. B 82 (2001) 175-222.
   \bibitem{Fries1927}
   K. Fries, Uber byclische verbindungen und ihren vergleich mit dem naphtalin, Ann. Chem. 454 (1927) 121-324.
   \bibitem{Kasteleyn1961}
   P.W. Kasteleyn, The statistics of dimers on a lattice: I. The number of dimer arrangements on a quadratic lattice, Physica 27 (1961) 1209-1225.
   \bibitem{klein1987innate}
   D.J. Klein, M. Randi\'c, Innate degree of freedom of a graph, J. Comput. Chem. 8 (1987) 516-521.
   \bibitem{Klein2014}
   D.J. Klein, V. Rosenfeld, Forcing, freedom, and uniqueness in graph theory and chemistry, Croat. Chem. Acta 87 (2014) 49-59.
   \bibitem{Lei2016Anti}
   H. Lei, Y. Yeh, H. Zhang, Anti-forcing numbers of perfect matchings of graphs, Discrete Appl. Math. 202 (2016) 95-105.
   \bibitem{Li1997Hexagonal}
   X. Li, Hexagonal systems with forcing single edges, Discrete Appl. Math. 72 (1997) 295-301.
   \bibitem{Little1974}
    C.H.C. Little, An extension of Kasteleyn's method of enumerating the 1-factors of planar graphs, in: D. Holton (Ed.), Combinatorial Mathematics, Proceedings 2nd Australian Conference, in: Lecture Notes in Mathematics, vol. 403, Springer, Berlin, 1974, pp. 63-72.
   \bibitem{Lovasz1983}
   L. Lov\'asz, Ear-decompositions of matching-covered graphs, Combinatorica 3 (1) (1983) 105-117.
   \bibitem{Lovasz1977}
   L. Lov\'asz, M.D. Plummer, On minimal elementary bipartite graphs, J. Combin. Theory Ser. B 23 (1977) 127-138.
   \bibitem{Lovasz1986}
   L. Lov\'asz, M.D. Plummer, Matching Theory, Ann. Discrete Math., Vol. 29, North-Holland, Amsterdam, 1986.
   \bibitem{NT2007}
   S. Norine, R. Thomas, Generating bricks. J. Combin. Theory Ser. B 97 (2007) 769-817.
   \bibitem{von1953}
   J. von Neumann, A certain zero-sum two-person game equivalent to an optimal assignment problem, Ann. Math. Studies 28 (1953) 5-12.
   \bibitem{RST1999}
   N. Robertson, P.D. Seymour, R. Thomas, Permanents, pfaffian orientations and even directed circuits, Ann. Math. 150 (1999) 929-975.
   \bibitem{shi2017maximum}
   L. Shi, H. Wang, H. Zhang, On the maximum forcing and anti-forcing numbers of (4, 6)-fullerenes, Discrete Appl. Math. 233 (2017) 187-194.
   \bibitem{vukivcevic2004total}
   D. Vuki\v cevi\'c, J. Sedlar, Total forcing number of the triangular grid, Math. Commun. 9 (2004) 169-179.
   \bibitem{vukiveevic2007anti}
   D. Vuki\v cevi\'c, N. Trinajsti\'c, On the anti-forcing number of benzenoids, J. Math. Chem. 42 (2007) 575-583.
   \bibitem{Yan2008}
   W. Yan, Y.-N. Yeh, F. Zhang, Dimer problem on the cylinder and torus, Physica A 387 (2008) 6069-6078.
   \bibitem{Zhang2022}
   Y. Zhang, H. Zhang, Relations between global forcing number and maximum anti-forcing number of a graph, Discrete Appl. Math. 311 (2022) 85-96.
   \bibitem{ZHLZ2025} Y. Zhang, X. He, Q. Liu, H. Zhang, Forcing, anti-forcing, global forcing and complete forcing on perfect matchings of graphs—A survey, Discrete
Appl. Math. 376 (2025) 318-347.
   \bibitem{ZS2019}
   S. Zhao, H. Zhang, Forcing and anti-forcing polynomials of perfect matchings for some rectangle grids, J. Math. Chem. 57 (2019) 202-225.
   \end{thebibliography}
\end{document}